\documentclass[12pt,reqno]{amsart}
\usepackage{amssymb,amsfonts}

\newcommand{\vertiii}[1]{{\left\vert\kern-0.25ex\left\vert\kern-0.25ex\left\vert #1
    \right\vert\kern-0.25ex\right\vert\kern-0.25ex\right\vert}}

\numberwithin{equation}{section}
\numberwithin{figure}{section}
\numberwithin{table}{section}
{ \theoremstyle{plain}

\newtheorem{theo}{Theorem}[section]
\newtheorem{lemm}[theo]{Lemma}
\newtheorem{coro}[theo]{Corollary}
\newtheorem{prop}[theo]{Proposition}
\newtheorem{rem}{Remark}

}

\def\beq{\begin{equation}}
\def\eeq{\end{equation}}

\def\R{\mathbb{R}}
\def\Z{\mathbb{Z}}
\def\C{\mathbb{C}}

\def\sm{\setminus}
\def\({\left(}
\def\){\right)}

\def\emin{{\underbar E}}
\def\ep{\varepsilon}
\def\pep{{\varepsilon'}}

\def\epp{{\varepsilon'}}
\def\eps{\varepsilon}
\def\te{{\tilde {\eps}}}
\def\tep{{\tilde {\eps}}}
\def\ptep{\tilde\varepsilon^\prime{}}
\def\lamep{{\lambda_\ep}}

\def\ltep{{|\log\tilde {\eps}|}}
\def\lptep{{|\log\tilde {\eps}'|}}

\def\la{\lambda}
\def\vphi{\varphi}
\def\O{\Omega}
\def\pO{{\Omega'}}

\def\om{\omega}
\def\pom{{\omega'}}

\def\oma{{\om_\alpha}}

\def\dm{\frac12}
\def\pjie{{j_{1,\ep}'}}
\def\pjje{{j_{2,\ep}'}}
\def\Jie{{j_{1,\ep}}}
\def\Jje{{j_{2,\ep}}}
\def\Jl{{j_{\beta}}}
\def\Jil{{j_{1,\beta}}}
\def\Jjl{{j_{2,\beta}}}
\def\rhoie{\rho_{1,\ep}}
\def\rhoje{\rho_{2,\ep}}
\def\rie{u_{1,\ep}}
\def\rje{u_{2,\ep}}
\def\prie{u_{1,\ep}'}
\def\prje{u_{2,\ep}'}
\def\vie{v_{1,\ep}}
\def\vje{v_{2,\ep}}
\def\viesq{{v^2_{1,\ep}}}
\def\vjesq{{v^2_{2,\ep}}}
\def\proie{{\rho_{1,\ep}'}}
\def\proje{{\rho_{2,\ep}'}}
\def\roie{\rho_{1,\ep}}
\def\roje{\rho_{2,\ep}}

\def\riep{u_{1,\epp}}
\def\rjep{u_{2,\epp}}
\def\ri{u_{1}}
\def\rj{u_{2}}
\def\mae{m_{\alpha,\ep}}
\def\Dep{\Delta_\ep}
\def\dep{d_\ep}
\def\depr{d_{\ep,r}}
\def\dept{d_{\ep,\theta}}
\def\mep{{m_\ep}}
\def\la{\ell_\alpha}

\def\tomt{{\tilde\om_t}}
\def\tUe{{\widetilde U_\ep}}
\def\vpie{\varphi_{1,\ep}}
\def\vpje{\varphi_{2,\ep}}
\def\hie{h_{1,\ep}}
\def\hje{h_{2,\ep}}
\def\nie{n_{1,\ep}}
\def\nje{n_{2,\ep}}
\def\muie{\mu_{1,\ep}}
\def\muje{\mu_{2,\ep}}
\def\muiie{{\mu^1_{i,\ep}}}
\def\mujie{{\mu^2_{i,\ep}}}
\def\aie{a_{i,\ep}}
\def\bie{b_{i,\ep}}
\def\np{{\nabla^\perp}}
\def\nue{{\nu_\ep}}
\def\thetaper{{\theta_\ep}}

\def\EOep{{E^\Omega_{\ep,\delta}}}
\def\pE{{E^\pO_{\pep,\delta}}}
\def\W{{W_\ep}}
\def\El{{J_\beta}}

\DeclareMathOperator{\per}{per}
\DeclareMathOperator{\dist}{dist}
\DeclareMathOperator*{\argmin}{argmin}
\DeclareMathOperator{\curl}{curl}
\DeclareMathOperator{\diver}{div}

\def\Xint#1{\mathchoice
   {\XXint\displaystyle\textstyle{#1}}%
   {\XXint\textstyle\scriptstyle{#1}}%
   {\XXint\scriptstyle\scriptscriptstyle{#1}}%
   {\XXint\scriptscriptstyle\scriptscriptstyle{#1}}%
   \!\int}
\def\XXint#1#2#3{{\setbox0=\hbox{$#1{#2#3}{\int}$}
     \vcenter{\hbox{$#2#3$}}\kern-.5\wd0}}
\def\dashint{\Xint-}

\begin{document}

\title[]{Vortex patterns and sheets in segregated Two Component Bose-Einstein condensates}

\author{Amandine Aftalion}
\address{Ecole des Hautes Etudes en Sciences Sociales, PSL Research University, CNRS UMR 8557, Centre d'Analyse et de Math\'ematique Sociales, 54 Boulevard Raspail, 75006 Paris, France.}
\email{amandine.aftalion@ehess.fr}

\author{Etienne Sandier}
\address{Universit´e Paris-Est,
LAMA - CNRS UMR 8050,
61, Avenue du G´en´eral de Gaulle, 94010 Cr´eteil, France.}
\email{sandier@u-pec.fr}

\date{\today}

\begin{abstract}
We study minimizers of a Gross--Pitaevskii energy
describing  a two-component Bose-Einstein condensate
  set into
rotation. We consider the case of segregation
of the components in the Thomas-Fermi regime, where a small
parameter $\eps$ conveys a singular perturbation.  We estimate the energy as a term due to a perimeter minimization and a term due to rotation. In particular, we prove a new estimate concerning the error of a Modica Mortola type energy away from the interface. For large rotations, we show that the interface between the components gets long, which is a first indication towards vortex sheets.
\end{abstract}

\maketitle

\section{Introduction}
In this paper, we study the vortex
structure in rotating immiscible two-component Bose Einstein condensates (BEC) in two dimensions. Indeed, when a two component condensate is set to high rotation, the ground state goes from a situation of segregation with vortices in each component, to a vortex sheet structure, as explained in \cite{AM,KT}.
 At zero rotation, the interface between the two components is given by a perimeter minimization similar to a Modica Mortola problem \cite{AAJRL2013,GM,GRL}. At higher rotation, there seems to be an interplay between perimeter minimization and vortex energy, leading possibly to a longer interface, as we will see below.
   A general numerical picture of the vortex states in rotating two component
  condensates is addressed by \cite{AM}:
the simulation of the coupled Gross-Pitaevskii
 equations are shown, discussing various configurations of the
vortex states, and, in the case of immiscible BECs, the vortex sheets with striped patterns,
the “serpentine” sheets, and the “rotating droplets.” The case of  droplets
 corresponds to two immiscible components, each having an individual vortex structure.
  The case of sheets is when the immiscible structure is at a lower scale than that of
  the condensates. The sheets can either be straight (stripes) or bent and connected (serpentines).
  There are other condensed-matter systems characterized
by  multicomponent order parameters in which
vortex sheets are observable \cite{helium3}.

 The two component condensate has
two interatomic coupling constants denoted by $g$ (for intracomponents), and $g_{12}$ (for intercomponent).
We confine ourselves to the phase-separated or segregation regime; in a
homogeneous system, the condition is given by $g_{12} > g$. For simplicity, we will set $g=1/\ep^2$
 and $\delta =g_{12} / g$.

  The ground state of a two component BEC is then described by two complex valued wave functions $u_1$ and $u_2$ defined in a domain $D$ of $\R^2$ minimizing
 the
following energy functional:
\begin{equation}\label{eqEnergyOmega}
%\begin{split}
E^\Omega_{\ep,\delta}(u_1,u_2)=\sum_{j=1}^2 \int_{D}\frac 12 |\nabla u_j-i\Omega x^{\perp} u_j|^2
 %\left\{ \frac{|\nabla u_j|^2}{2}  -\Omega x^{\perp}\cdot(i u_j,\nabla u_j) \right\} \,dx \\
 +\int_{D} W_{\ep,\delta} (|u_1|^2,|u_2|^2) \,dx
%\end{split}
\end{equation}

where \beq\label{Wdelta}W_{\ep,\delta} (|u_1|^2,|u_2|^2)=
\frac 1 {4\ep^2} (1-|u_1|^2)^2 +\frac 1 {4\ep^2} (1-|u_2|^2)^2+\frac \delta {2\ep^2} |u_1|^2 |u_2|^2
-\frac 1 {4\ep^2},\eeq
 in the space
\begin{equation}\label{H_space_definition}
\mathcal{H}=\left\{(u_1,u_2):\ u_j\in H^1(D,\C), \
  \dashint_D |u_j|^2={\alpha_i},\
j=1,2\right\},
\end{equation} where $\dashint_D |u_j|^2=\int_D |u_j|^2 /|D|$.
The parameters $\delta,\ep$ and $\Omega$ are positive: $\Omega$ is the
angular velocity corresponding to the rotation of the condensate, $x^\perp=(-x_2,x_1)$.
 %and $\cdot$ is the scalar product for vectors, whereas $(\ ,\ )$ is the complex scalar product, so that  we have
%\[
%x^{\perp}\cdot(iu,\nabla u)=x^{\perp}\cdot\frac{iu\nabla\bar u-i\bar
%u\nabla u}{2}= -x_2\frac{iu\partial_{x_1}\bar u-i\bar
%u\partial_{x_1}u}{2}+x_1\frac{iu\partial_{x_2}\bar u-i\bar
%u\partial_{x_2}u}{2}.
%\]

 We are interested in studying the existence and behavior of the
minimizers in the limit when $\ep$ is
small, describing strong interactions, also called the Thomas-Fermi limit.

The potential term can be rewritten as
\beq\label{Wdeltabis}W_{\ep,\delta} (|u_1|^2,|u_2|^2)=\frac 1 {4\ep^2} (1-|u_1|^2-|u_2|^2)^2
+\frac {(\delta-1)} {2\ep^2} |u_1|^2 |u_2|^2.\eeq
 We focus on the regime where $(\delta -1)$ is small, like a power or function of $\ep$.
  We expect that in this limit, $(1-|u_1|^2-|u_2|^2)$ and $|u_1| |u_2|$ tend to
   zero, probably on different scales.

   In particular, we want to estimate the energy in order to understand the vortex patterns.

  In order to understand the $\Gamma$-limit of $E_{\ep,\delta}^\Omega$, one needs to understand on the
  one hand the behaviour at $\O=0$ (no rotation) which provides a  perimeter minimization problem, and on the other hand the influence of rotation on the vortex structure.

  At $\O=0$, the problem is real valued. In the limit when $\ep$ tends to 0, the domain $D$ is divided into two domains $D_1$ et $D_2$, s.t. $|D_1|=\alpha_1 |D| $, $|D_2|=\alpha_2 |D| $, and the length of $\partial D_1 \cap D$ is minimized. More precisely,
  for a pair of real valued functions $\ri, \rj:D\to\R$, let
  \beq\label{scalaren}
  F_{\ep,\delta}(\ri, \rj):= \int_D \dm\(|\nabla\ri|^2+ |\nabla\rj|^2\) + W_{\ep,\delta} (u_1^2,u_2^2).\eeq
  We also define for any given $\alpha\in (0,1)$
  \beq m_{\alpha,\ep,\delta} = \min\left\{F_{\ep,\delta}(\ri, \rj)\mid \ri,\rj\in H^1(D,\R), \dashint_D|\ri|^2 = \alpha, \dashint_D|\rj|^2 = 1 - \alpha\right\}\eeq and
\beq\label{minper}\la = \min_{\substack{\om\subset D\\ |\om| = \alpha |D|}} \per_D(\om).\eeq The segregation problem has been studied by many authors \cite{BeLinWeiZhao,BeTer,CaffLin,CaffLin2,SoZ,sourdis}. There are results about the regularity and connectedness, and the fact that the interface goes from one part of the boundary to another \cite{AF,ros,SZ}. There are also results about the $\Gamma$ limit \cite{AAJRL2013,GRL,GM} which rely on similar techniques to those used for the Mumford Shah functional \cite{AFP,AT}.

  The order of magnitude of $\delta$ has a strong impact on $m_{\alpha,\ep,\delta}$ and
  the boundary layer between the two components. Let $v_\ep^2=u_1^2+u_2^2$. Then $v_\ep^2$ tends to 1 in each component but on the boundary between the two components,  the behaviour of $v_\ep$ depends on $\delta$. More precisely,
 \begin{itemize}\item if $\delta$ tends to $\infty$, then $\inf v_\ep$ goes to 0 (see \cite{AAJRL2013}) and the
 $\Gamma$-limit of $\ep F_{\ep,\delta}$ is
 $$c {\la}$$ where  $c$ is an explicit constant corresponding to the Modica Mortola phase transition problem, and $l_\alpha$ is given by \eqref{minper}. \item if $\delta$ is of order 1, then $\inf v_\ep$ tends to some number  between 0 and 1 and the
  $\Gamma$-limit of $\ep F_{\ep,\delta}$ is
 $$c_\delta  {\la}$$ where $c_\delta>0 $ depends on $\delta$ (see \cite{GRL}).
  \item if $\delta$ tends to $1$ as $\ep \to 0$,  then the $\Gamma$-limit of $\displaystyle\frac\ep{\sqrt{\delta -1}} F_{\ep,\delta}$ is
 $$  \la/2$$    as proved in \cite{GM}, and we expect that $\inf v_\ep$ tends to 1, though a refined convergence is still missing.\end{itemize}

  When $\O$ increases from 0, we expect that the next order term in the energy will depend on the existence of vortices in the system. For a one component condensate, the rotating case is based on the work of \cite{SSbook} and has been detailed in \cite{Serf} (see also \cite{Abook,IM1,IM2}). The main features are that there exists a critical value $\O_1$ of the rotational velocity of order $\ln 1/\ep$ under which no vortices are present in the system and the energy is of order $\O^2$. For $\O\gg\O_1$, the system has a uniform density of vortices and the energy is of order $\O\log (1/\ep\sqrt\O)$. For a two component condensate in the coexistence  regime ($\delta <1$), the absence of vortices up to the first critical velocity has been proved in \cite{ANS}.

  In the segregating regime ($\delta >1$), the analysis is totally open. Nevertheless, we expect that the minimization of the energy decouples. On the one hand, there is  the minimization of the interface energy, that is the length of the perimeter of the boundary between the two regions occupied by each component. On the other hand, there is a minimization of the vortex energy in each region, similar to the case with one condensate, which may lead to a vortex structure in each region. In fact, simple calculations show that these energies have different orders of magnitude.

 When $\delta$ tends to 1, the effective length scale of  the phase  transition  and of the size of the vortex cores is $\tilde \ep= \ep/\sqrt{\delta -1}$.
 Therefore, the critical velocity for the nucleation of vortices is expected to be
 $$\O_1 =c_1 \log \frac 1 {\tilde \ep}.$$ Moreover, vortices should exist up to $\O_2=c_2/\tilde\ep^2$.

 \begin{rem} We have made the choice to include a complete square in the first term of the energy
  without subtracting the centrifugal term $\Omega^2 x^2 u^2$, which for $\Omega^2\ep^2\ll 1$ leads to the same energy expansion and vortex patterns. At $\O=1/\eps$, as explained in \cite{CPRY}, the energy without the centrifugal term displays a change of
 behaviour: the bulk of the condensate becomes annular. The two energy yield the same structures for
 rotationnal velocities much lower than $1/\eps$; in the case when $\delta$ tends to 1, velocities up to $1/\tep$ can be less than $1/\ep$ if  $\ep\ll\tep^2$.\end{rem}

 Since we are going to assume that $\delta$ tends to 1 as $\ep$ tends to 0, we remove the dependencies in $\delta$ and define for any given $\alpha\in (0,1)$ and $\ep>0$, the energy without rotation of a pair $\ri, \rj:D\to\R$ by
  \beq\label{scalar_energy}
  F_\ep(\ri, \rj):= \int_D \dm\(|\nabla\ri|^2+ |\nabla\rj|^2\) + \W(\ri,\rj),\eeq
  where
  \beq\label{vep1}
  \quad \W(\ri,\rj) = \frac1{4\ep^2} (1 - |\ri|^2)^2 + \frac1{4\ep^2} (1 - |\rj|^2)^2 + \frac{\delta}{2\ep^2}  |\ri|^2|\rj|^2 - \frac1{4\ep^2}.\eeq
  Moreover we let
 \beq\label{mae}\mae = \min\left\{F_\ep(\ri, \rj)\mid \ri,\rj\in H^1(D,\R), \dashint_D|\ri|^2 = \alpha, \dashint_D|\rj|^2 = 1 - \alpha\right\}.\eeq
 It follows from \cite{GM} that $\mae$ is of order $\sqrt{\delta-1}/\ep$ when $\delta$ tends to 1 hence  of order $1/\tep$. The  relation between $\la$ given by \eqref{minper} and $\mae$ is well-known since the work of Modica-Mortola \cite{momo} for a similar functional.
   More precisely,  $\mae\sim \ell_\alpha \mep$, where
 \beq \label{me}\mep :=  \inf_{\substack{\gamma:\R\to\R^2\\ \gamma(-\infty) = (1,0)\\ \gamma(+\infty) = (0,1)}}\int_{-\infty}^{+\infty} \frac{|\gamma'(t)|^2}{2}+ \W(\gamma(t))\,dt.\eeq
 Note that $\mep$ depends on $\ep$ and $\tep$ but is equivalent to $1/2\tep$ at  leading order as proved in \cite{GM}.

 Our main result about the energy expansion and the vortex pattern is the following:

\begin{theo}\label{thmin} Assume $D$ is a smooth bounded domain in $\R^2$ and that $\alpha\in (0,1)$. Recall that $\EOep$ is defined by \eqref{eqEnergyOmega}, where $\delta = \delta(\ep)$ and $\Omega = \Omega(\ep)$,  and assume $\tep = \ep/\sqrt{\delta-1}$ is such that $\tep\to 0,$ $\tep\ll\ep $ as $\ep\to 0$. Let $u_\ep =(\rie,\rje)$ denote a minimizer of $\EOep$ under the constraint
\beq\label{constraint} \dashint_D |\rie|^2 = \alpha,\quad\dashint_D |\rje|^2 = 1- \alpha.\eeq
Then the following behaviours hold, according to different rotation regimes:

\begin{description}
\item[A] If $\Omega/\ltep$ converges to $\beta\ge 0$ then  $\(|\rie|,|\rje|\)$ converges weakly in $BV$ to $(\chi_{\om_\alpha}, \chi_{\om_\alpha^c})$, where $\om_\alpha$ is a minimizer of $\per_D(\om)$ under the constraint $|\om| = \alpha |D|$.

Moreover, let
\beq\label{jeps}\Jie = (i\rie,\nabla\rie)-\Omega x^\perp |\rie|^2, \quad \Jje = (i\rje,\nabla\rje)-\Omega x^\perp |\rje|^2,\eeq
then  $(\Jie/\Omega,\Jje/\Omega)$ converges weakly in $L^2$ to $(\Jil, \Jjl)$, where
\beq\label{jlambda} \Jil = \argmin_{\diver j = 0} \El(j,\om_\alpha),\quad \El(j,\om_\alpha) = \dm\int_{\om_\alpha} |j|^2 +\frac1{2\beta}\int_{\om_\alpha}\left|\curl j +2\right|,\eeq
and $\Jjl$ is defined similarly, replacing $\om_\alpha$ by $\om_\alpha^c$. In the case $\beta = 0$ we have to interpret the definition of $\El(j,\om)$ as follows: it is equal to  $\|j\|_{L^2(\om)}^2$ if $\curl j +2 = 0$, and to $+\infty$ otherwise. Moreover

\beq\label{minA} \min_{\mathcal H} \EOep = \mep\la + \Omega^2 \(\min_{\diver j = 0} \El(j,\om_\alpha)+ \min_{\diver j = 0} \El(j,\om_\alpha^c)\) + o(\ltep^2).\eeq

\item[B] If $\ltep\ll \Omega$ and $\Omega \log\(1/\tep\sqrt\Omega\)\ll 1/\tep$ then $\(|\rie|,|\rje|\)$ still converges as above to $(\chi_{\om_\alpha}, \chi_{\om_\alpha^c})$ and, defining $\Jie$, $\Jje$ as in \eqref{jeps}, both $\Jie/\Omega$ and $\Jje/\Omega$ converge weakly to $0$ in $L^2$. Moreover
\beq\label{minB} \min_{\mathcal H} \EOep = \mep\la + \dm |D| \Omega \log\(\frac1{\tep\sqrt\Omega}\) \(1+o(1)\).\eeq
\item[C] If $1/\tep \ll \Omega \log\(1/\tep\sqrt\Omega\)\ll 1/\tep^2$ then
\beq\label{minC} \min_{\mathcal H} \EOep = \dm |D| \Omega \log\(\frac1{\tep\sqrt\Omega}\)  \(1+o(1)\).\eeq
\end{description}
\end{theo} In cases {\bf A} and {\bf B}, the leading order term is the interface energy $\mae$ which is of order $1/\tilde{\ep}$. This leads to two droplets having individual vortices. This interface term stays dominant until $\Omega \log\(1/\tep\sqrt\Omega\)$ reaches $1/\tep$.  For high rotations, we do not know if  the interface still minimizes the perimeter, but we believe that the interface is allowed to increase its length to reach the sheet pattern.
Note that the hypothesis $\ep^2\ll \tep$ guarantees that $\Omega$ must be less than $1/\ep$.

The proof of the above theorem builds upon the analysis of Ginzburg-Landau vortices in the presence of a magnetic field (see \cite{ss1,ss2,js1,js2} or the book \cite{SSbook}). The problem here is to factor out the energy of the  interface between the set $\om_\alpha$ where $|\rie|\simeq 1, |\rje|\simeq 0$  and the set $\om_\alpha^c$ where $|\rie|\simeq 0, |\rje|\simeq 1$. In cases A and B of the Theorem, this interface energy is dominant hence it is difficult to separate it from the vortex energy which is computed separately in each domain $\om_\alpha$, $\om_\alpha^c$. Note that we cannot separate this leading-order energy by a splitting argument as in the Ginzburg-Landau case or using the division trick introduced in \cite{LaMi} and used since in different contexts (see \cite{andreshafrir,IM1,kachmar} for instance) because of the segregation pattern: one component has an almost zero density.

We rely instead on the fact that the interface energy is due to the modulus of $\rie$ and $\rje$, while the vortex energy is due to the phase. The argument requires nevertheless to precisely locate the interface energy  and estimate  the rest of the energy away from the interface, as we will see in Theorem~\ref{enloc} below. This is a result which to our knowledge is new even in the case of the Modica-Mortola functional. A more precise lower bound  was proved by G.Leoni and R.Murray \cite{LM} but without locating the energy.

 \begin{theo} \label{enloc} Let $D$ be a bounded smooth domain in $\R^2$ and $\alpha\in(0,1)$. Assume  $\delta = \delta(\ep)$ is such that $\delta$ tends to $1$ and $\te := \frac\ep{\sqrt{\delta - 1}}$ tends to $0$,  as $\ep\to 0$. Denote by $\{\ep\}$  a sequence of real numbers tending to $0$.

Let $\{(\rie, \rje)\}_\ep$ be such that
\beq\label{hyptheo} F_\ep(\rie, \rje)\le \mep\la + \Dep, \quad \rie^2+\rje^2  \le 1+ C\tep, \eeq
where $\mep$ is given by \eqref{me} and $\la$ is given by \eqref{minper}, with $\Delta_\ep \ll \mep\la$ as $\ep\to 0$. Then there exists a subsequence $\{\epp\}$ such that $\{(\riep, \rjep)\}_\epp$ converges to $(\chi_{\om_\alpha}, \chi_{\om_\alpha^c})$, where $\om_\alpha$ is a minimizer of \eqref{minper}.

Moreover writing  $\gamma_\alpha = \partial\om_\alpha\cap D$, for any $\eta>0$ there exists $C>0$ such that if $\epp$ is small enough (depending on $\eta$), for any $V_\eta$ which is  an $\eta$-neighbourhood of $\gamma_\alpha$ we have
\beq\label{uppb} \mep\la - C(\Dep +\ltep) \le F_\ep(\riep, \rjep,V_\eta),\quad F_\ep(\riep, \rjep,D\setminus V_\eta) \le C (\Dep +\ltep).\eeq
\end{theo}
The hypothesis $\rie^2+\rje^2  \le 1+C\tep$  is satisfied (see Proposition~\ref{proplinfinitybound} below) for minimizers of $\EOep$ if $\Omega$ is not too large, as in cases~A and~B of Theorem~\ref{thmin}. In case~C it does not apply, but in this case the leading order of the energy does not allow to locate the interface anyway.

Theorem~\ref{enloc} means that the energy is concentrated close to the interface up to an error of order $\ltep$. The proof will follow from  a similar concentration of perimeter for problem \eqref{minper} and by estimating $\mae$ in terms of perimeters of level-sets of a certain function, as in P.Sternberg's \cite{sternberg} generalization of the method of Modica-Mortola \cite{momo}.

The proof of Theorem \ref{thmin} relies on precise upper bounds and lower bounds. The upper bound consists in building a test function whose modulus approaches the interface problem and whose phase reproduces the expected pattern for vortices depending on the values of $\Omega$. One difficulty is that we have to keep the mass constraint satisfied and $|u_1|^2+|u_2|^2$ close to $1$. An important tool is the uniform exponential decay when $\delta$ tends to 1, proved for the 1D problem in \cite{sourdis}. Let us point out that we have chosen the limit $\delta\to 1$ because it is only in this case of weak separation that the sheets exist. In the case where $\delta$ is fixed the interface problem leads to two domains having their own vortices and the proof can be adapted from what we have done.

\hfill

When $\Omega$ is of the order of $1/\tep^2$, assuming $1/\tep^2\ll 1/\ep$, we are no longer able to determine the leading order of the minimal energy. However a plausible minimizer exists, neglecting boundary effects, which depends on one variable only and exhibits a stripe pattern. The construction yields the following

\begin{theo} \label{upstripes}Assume that $\Omega = \lambda/\tep^2$ and that $\ep\ll \tep^2$, then
\beq\label{big} \min_{\mathcal H} \EOep \le \Omega |D| \emin (\alpha,\lambda),\eeq
where
\beq\label{emin} \emin(\lambda,\alpha) = \min_{\mu>0} \min_{\theta\in X_\alpha}
\left \{\(\frac 16+\frac\alpha 2 - 4 \int_0^{1/2}x\sin^2\theta\) \mu^2+\frac 1 {\mu^2} \int_0^{1/2}{\theta '}^2+\frac 1{4 \lambda} \int_0^{1/2} \sin^2 2\theta\right\},\eeq
and $X_\alpha$ is the set of $\theta$ in $H^1(0,1/2)$ such that
\beq\label{X} \dashint_0^{1/2} {\sin}^2 \theta = \alpha,\quad \theta(0) = 0,\quad \theta(1/2) = \pi/2.\eeq
\end{theo}

\begin{rem}
%It is straightforward to expand $\emin(\alpha,\lambda)$ when $\lambda$ is large or small. For  $\lambda$  large, we let $c_\alpha$ be the minimum of   $\int_0^{1/2} {\theta '}^2$ over $X_\alpha$. Then, ignoring  the term $\int_0^{1/2} {\sin^2 2\theta }$ and optimizing with respect to $\mu$ we find that $\emin(\alpha,\lambda) \simeq 2\sqrt{c_\alpha/6}$.

 %On the contrary,
 If $\lambda$ is small,  the $\theta$ energy is  of Modica Mortola type and $\theta$ varies quickly from 0 to $\pi/2$ on a scale $\sqrt\lambda/\mu$. In this case $\sin\theta = 0$ except on the transition interval therefore the term $\int_0^{1/2}x\sin^2\theta$ can be neglected in front of the constant terms. Optimizing with respect to $\theta$ yields, to first order as $\lambda\to 0$, $\mu^2(1/6+\alpha/2) + c_0/(\mu\sqrt\lambda)$. Optimizing with respect to $\mu$ then yields  that $\emin(\alpha,\lambda)$  is of order $1/\lambda^{1/3}$. Note however that in this regime of small $\lambda$, Theorem~\ref{thmin}, case~C shows that this upper-bound is not optimal.
 \end{rem} An alternative direction of construction of upper bounds could be the framework developped by \cite{wei} for two species polymers.

Still in this regime, one thing we are able to say about minimizers $(\rie,\rje)$ is that on most disks of radius  $R\tep$, both $\rie$ and $\rje$ are present. More precisely,

\begin{theo}\label{theolong} Assume that $\Omega = \lambda/\tep^2$ and that $\ep\ll \tep^2$, then
 for all $\eta>0$, there exists a $\beta >0$, $R_0>0$, such that  for $R>R_0$, and
 for all $\eps$ sufficiently small, if $(u_1,u_2)$ is a minimizer of $\EOep$ in ${\mathcal H}$, then
   \beq\label{estl2norm} |\{ x \hbox{ s.t. } \dashint_{D(x,R,\tep)} |u_1|^2 < \beta \hbox{ or }
  \dashint_{D(x,R\tep)} |u_2|^2 < \beta \} |<\eta\eeq where $D(x,R\tep)$ is the circle of center $x$ and radius $R\tep$.
\end{theo}

The paper is organized as follows. In section 2, we prove estimates that will be useful all along the proofs, namely an $L^\infty$ estimate, estimates for the corresponding 1D problem and relations between the minimum for the 2D and 1D problems. Section 3 is devoted to the proof of Theorem \ref{thmin} assuming that Theorem \ref{enloc} holds: upper bounds and lower bounds are built carefully leading eventually to the required energy estimates. In Section 4, we introduce the perimeter related properties that allows us to eventually prove Theorem \ref{enloc}. The last section deals with the sheets case leading to the proofs of Theorems \ref{upstripes} and \ref{theolong}.

\section{A priori estimates}
A minimizer of \eqref{eqEnergyOmega} in $\mathcal{H}$ given by \eqref{H_space_definition} is a solution of the following system, \begin{subequations}\label{eq:mainu1u2}
\begin{align} \label{eq:mainu1u2a}
&- \Delta u_1-2 i\Omega x^\perp\cdot \nabla u_1 +\frac 1 {\ep^2}u_1(|u_1|^2+|u_{2}|^2-1+\ep^2\Omega^2 |x|^2)+\frac {(\delta-1)} {\ep^2}|u_2|^2 u_1=\lambda_{1}u_1, \\
 \label{eq:mainu1u2b}
&- \Delta
u_2-2 i\Omega x^\perp\cdot \nabla u_2 +\frac 1 {\ep^2}u_2(|u_1|^2+|u_{2}|^2-1+\ep^2\Omega^2 |x|^2)+\frac {(\delta-1)} {\ep^2}|u_1|^2 u_2 =\lambda_{2}u_2,
\end{align}
\end{subequations}where $\lambda_j$'s are the Lagrange multipliers due to the $L^2$ constraint.

\subsection{$L^\infty$ estimates}

In order to get an a priori estimate for $w=|u_1|^2+|u_2|^2$ using the equation satisfied by $w$,  we need to prove that the Lagrange multipliers are positive.

\begin{lemm}If $(u_1,u_2)$ is a minimizer of \eqref{eqEnergyOmega} in $\mathcal{H}$, then the Lagrange multipliers $(\lambda_1,\lambda_2)$ in equations \eqref{eq:mainu1u2} are nonnegative.\end{lemm}\begin{proof}
We multiply \eqref{eq:mainu1u2a} by $\bar u_1$, integrate and add the complex conjugate to find
\beq\label{eqla1}\lambda_1\alpha_1=\int_D |\nabla u_1-i\Omega x^{\perp} u_1|^2+\frac {|u_1|^2} {\ep^2} (|u_1|^2+|u_2|^2-1
%-\ep^2\Omega^2|x|^2
)+\frac {(\delta-1)} {\ep^2} |u_1|^2 |u_2|^2.\eeq The corresponding equation holds for $\lambda_2$.
If one computes the second variation of the energy at a minimizer $(u_1,u_2)$ against functions $\varphi$:
\begin{multline*}\frac{\partial^2 E^\Omega_\ep(u_1,u_2)}{{\partial u_1}^2}\cdot (\varphi,0)=\int_D \frac 12|\nabla \varphi-i\Omega x^{\perp} \varphi|^2 +\frac {|\varphi|^2}{2\ep^2}{(|u_1|^2+|u_2|^2-1)}+\\ +\frac {\delta-1} {2\ep^2} |\varphi|^2 |u_2|^2+\frac {(\bar u_1\varphi+u_1\bar \varphi)^2}{4\ep^2}\end{multline*}
 If we assume
\beq\label{condorth}\int \bar{u}_1\varphi+u_1\bar{\varphi}=0,\eeq then this second variation is nonnegative, since we are at a minimizer. It turns out that if one takes $\varphi=iu_1$, then it satisfies pointwise $\bar{u}_1\varphi+u_1\bar{\varphi}=0$, and therefore the expression for $\lambda_1$ \eqref{eqla1} is exactly this second variation, hence is nonnegative. The same works out for $u_2$ and $\lambda_2$. \end{proof}

\begin{lemm}If $(u_1,u_2)$ is a minimizer of \eqref{eqEnergyOmega} in $\mathcal{H}$, then $(\lambda_1,\lambda_2)$ in equations \eqref{eq:mainu1u2} satisfy
\beq\label{estlaj}\lambda_j\leq \frac4{\alpha_j} \EOep(u_1,u_2),\quad\text{where}\quad \alpha_j = \dashint_D |u_j|^2.\eeq\end{lemm} \begin{proof}We add \eqref{eqla1} and the corresponding equation for $\lambda_2$ to find
$$\alpha_1\lambda_1+\alpha_2\lambda_2= \int_{D} \sum_{j=1}^2|\nabla u_j-i\Omega x^{\perp} u_j|^2+\frac {|u_1|^2+|u_2|^2} {\ep^2} (|u_1|^2+|u_2|^2-1
%-\ep^2\Omega^2|x|^2
)+\frac {2(\delta-1)} {\ep^2} |u_1|^2 |u_2|^2$$ Since we have the $L^2$ constraint, and $\alpha_1+\alpha_2 =1$, then $\int |u_1|^2+|u_2|^2=1$, therefore,
$$\alpha_1\lambda_1+\alpha_2\lambda_2= \int_{D} \sum_{j=1}^2|\nabla u_j-i\Omega x^{\perp} u_j|^2+\frac 1 {\ep^2} (|u_1|^2+|u_2|^2-1)^2
%-\ep^2\Omega^2|x|^2
+\frac {2(\delta-1)} {\ep^2} |u_1|^2 |u_2|^2. $$ Since the $\lambda_i$'s are nonnegative, the result follows.\end{proof}

\begin{prop}\label{proplinfinitybound}If $(u_1,u_2)$ is a minimizer of \eqref{eqEnergyOmega} in $\mathcal{H}$, then
\beq\label{estw}\max (|u_1|^2+|u_2|^2) \leq 1+C\ep^2 \EOep(u_1,u_2).\eeq
\end{prop}\begin{proof}We look for the equation satisfied by $w=|u_1|^2+|u_2|^2$: we multiply \eqref{eq:mainu1u2a} by $\bar u_1$, add the complex conjugate, and add the corresponding term with $u_2$ to find
$$\Delta w=2\sum_{j=1}^2|\nabla u_j-i\Omega x^{\perp} u_j|^2-2\lambda_1 |u_1|^2-2\lambda_2 |u_2|^2+\frac 2{\ep^2} w(w-1)+\frac {2(\delta-1)} {\ep^2} |u_1|^2 |u_2|^2.$$
This leads to $$\Delta w \geq \frac 2{\ep^2} w(w-1- \ep^2 \max (\lambda_1,\lambda_2)),$$ which implies
$$\max w\leq 1+\ep^2 \max (\lambda_1,\lambda_2).$$ The previous Lemma yields the result.\end{proof}

\subsection{the 1D system}
\begin{prop}\label{propexpdecrease}There exists a unique minimizer of
 \beq\label{minen1d}\int_{-\infty}^{+\infty} \dm|v_1'|^2+\dm |v_2'|^2  + \frac 1{4\ep^2}
 (1-|v_1|^2- |v_2|^2)^2+ \frac{\delta -1}{2\ep^2}|v_1|^2|v_2|^2\,dx.\eeq \begin{equation}\label{eqBdryGenD}
	(v_1,v_2)\to \left(0,1\right)\ \textrm{as}\ x\to -\infty,\ \ (v_1,v_2)\to
	\left(1,0\right)\ \textrm{as}\ x\to  +\infty.
	\end{equation} Moreover $|v_1|^2+ |v_2|^2\leq 1$ and \beq 0\leq \int_{-\infty}^{+\infty} (1-|v_1|^2- |v_2|^2)\leq C\frac {\ep^2}{\tilde \ep}.\eeq
    \end{prop} \begin{proof}It follows from \cite{alama}, Theorem 3.1, that there exists a minimizer for problem   \eqref{minen1d}-\eqref{eqBdryGenD}. Moreover, each minimizer satisfies that each component is monotone. Therefore, it follows from the results of uniqueness of \cite{AS} for the solutions of the corresponding Euler-Lagrange equations with monotone components that the minimizer is unique.

    The minimizer is a solution of
    \begin{equation}\label{sys1DD}
	\left\{
	\begin{array}{c}
	-v_1''+\frac 1{\ep^2} v_1(v_1^2+v_2^2- 1)+ \frac{\delta -1}{\ep^2}v_2^2 v_1=0, \\
	\\
	-v_2''+\frac 1{\ep^2}v_2  (v_1^2+v_2^2- 1)+ \frac{\delta -1}{\ep^2} v_1^2 v_2=0,
	\end{array}
	\right.
	\end{equation} In order to prove that $|v_1|^2+ |v_2|^2\leq 1$, we define $w=|v_1|^2+|v_2|^2$ and compute the equation satisfied by $w$ which yields $$w''\geq \frac 2{\ep^2} w(w-1)$$ and implies that the maximum of $w$ is less than 1.

Then, we follow the Pohozaev type proof and multiply the first equation of \eqref{sys1DD} by $x v_1'$, the second by $xv_2'$ and integrate to find
\beq\int_{-\infty}^{+\infty} \dm |v_1'|^2+\dm |v_2'|^2=\int_{-\infty}^{+\infty}\frac 1{4\ep^2}
 (1-|v_1|^2- |v_2|^2)^2+ \frac{\delta -1}{2\ep^2}|v_1|^2|v_2|^2. \eeq Moreover, if we multiply the first equation of \eqref{sys1DD} by $v_1$, the second by $v_2$, integrate and add the 2, we find\beq\int_{-\infty}^{+\infty} \frac 14 |v_1'|^2+\frac 14 |v_2'|^2 + \frac 1{4\ep^2}
 (|v_1|^2+|v_2|^2-1)(|v_1|^2+|v_2|^2)+ \frac{\delta -1}{2\ep^2}|v_1|^2|v_2|^2=0.\eeq Subtracting the two, we find
 \beq\label{estintu1u2}\int_{-\infty}^{+\infty} \frac 32 |v_1'|^2+\frac 32 |v_2'|^2= \int_{-\infty}^{+\infty}\frac 1{2\ep^2} (1-|v_1|^2- |v_2|^2). \eeq
 The energy estimate provides the result.\end{proof}
The next result is about the decrease at infinity for the rescaled 1D system:
\begin{prop}\label{propexpdecrease2} If  \eqref{sys1DD} is rescaled by $\tilde\ep$ then the new  system is  \begin{equation}\label{sys1D}
	\left\{
	\begin{array}{c}
	-v_1''+\frac 1{\delta -1} v_1(v_1^2+v_2^2- 1)+ v_2^2 v_1=0, \\
	\\
	-v_2''+\frac 1{\delta -1}v_2 (v_1^2+v_2^2- 1)+ v_1^2 v_2=0,
	\end{array}
	\right.
	\end{equation}\begin{equation}\label{eqBdryGen}
	(v_1,v_2)\to \left(0,1\right)\ \textrm{as}\ x\to -\infty,\ \ (v_1,v_2)\to
	\left(1,0\right)\ \textrm{as}\ x\to  +\infty.
	\end{equation}  The solutions converge exponentially fast to its limit at $\pm \infty$, uniformly in $\ep$.\end{prop}
\begin{proof}It follows from \cite{alama}, Theorem 3.1, that there exists a minimizer for problem   \eqref{mep}. Moreover, each minimizer satisfies that each component is monotone. Therefore, it follows from the results of uniqueness of \cite{AS} for the solutions of the corresponding Euler-Lagrange equations with monotone components that the minimizer is unique. The exponential convergence at infinity is a consequence of the results of \cite{sourdis}.

	 To follow the results of \cite{sourdis}, the system can be expressed in polar coordinates:
$$v_1=R \sin\varphi_1,\quad  v_2 =R \cos\varphi_1.$$ In order to apply the slow fast theory, one considers the small parameter $\sqrt{\delta -1}$ and rewrite $R=1-(\delta -1) w_1$. Then writing $w_2=w_1'$ and $\varphi_2=\varphi_1 '$, system \eqref{sys1D} can be rewritten as a first order system in $(w_1,w_2,\varphi_1,\varphi_2)$. The results of \cite{sourdis} imply that
\begin{eqnarray}\label{behavinfty}
	w_1=\frac {e^{2x}}{(1+e^{2x})^2}+O(\sqrt{\delta -1}) \min (e^{2x},e^{-2x}) \\
	\varphi_1=\arctan e^x+O(\sqrt{\delta -1}) \min (e^{x},e^{-x})
	\end{eqnarray} uniformly as $\delta\to 1$.
 This implies the uniform exponential convergence at infinity for the functions $v_1$ and $v_2$.\end{proof}
%{\bf AA: il semble qu'il faille supposer pour ne pas avoir de souci que $\Omega^2 \ll 1/ {\tilde \ep}$. 

\subsection{Upper bound for the scalar problem}

From now on, $\delta(\ep)$ is such that $\lim_{\ep\to 0}\delta = 1$ and
$$\lim_{\ep\to 0} \tep = 0,\quad\text{where}\quad \tep := \frac\ep{\sqrt{\delta(\ep) - 1}}.$$
Therefore the potential $\W$ only depends on $\ep$ and is defined by
\beq \label{vep}\W(u_1,u_2)=\frac 1 {4\ep^2} (1-{u_1}^2-{u_2}^2)^2
+\frac {(\delta-1)} {2\ep^2} {u_1}^2 {u_2}^2,\eeq

Firstly we define
\beq \label{mep} \mep = \inf\left\{\int_{-\infty}^{+\infty} \dm|\gamma'(t)|^2 + \W(\gamma(t))\,dt\mid \gamma:\R_+\to\R^n,\ \lim_{-\infty}\gamma = a,\ \lim_{+\infty}\gamma = b\right\},\eeq
where $a=(1,0)$ and $b = (0,1)$ are the two wells of the potential $\W$.

The following upper-bound is proved using a standard construction, found for instance in \cite{GM} in this particular case, but with a less precise estimate.

\begin{prop}\label{propupscalar} Assume $\alpha\in(0,1)$, $D$ is a smooth bounded domain, and let $\mae$, $\la$ be defined in \eqref{mae}, \eqref{minper}. There exists $C>0$ such that for any small enough $\ep>0$, the following estimate holds:
\beq\label{bmae}\mae\le \la \mep+ C.\eeq
Moreover, let $\gamma_\alpha = \partial\om_\alpha\cap D$, where $\om_\alpha$ is a minimizer for \eqref{minper}, be  a minimal interface. Then  for any $\eta>0$, and denoting by $V_\eta$ an $\eta$-neighbourhood of $\gamma_\alpha$,  the above bound may be achieved by  $v_\ep = (v_{1,\ep}, v_{2, \ep}):D\to\R_+\times\R_+$ such that if $\ep$ is small enough then
\begin{equation}\label{estvv}v_{1,\ep} = 1 \hbox{ in }\om_\alpha\sm V_\eta,\ v_{1,\ep} = 0 \hbox{ in } {\om_\alpha}^c\sm V_\eta, \ v_{2,\ep} = 0 \hbox{  in } \om_\alpha\sm V_\eta \hbox{ and }\|v_{2,\ep}- 1\|<C\ep \hbox{ in } C^1({\om_\alpha}^c\sm V_\eta).\end{equation}
\end{prop}

\begin{rem} The constant $C$ in \eqref{bmae} depends  on $D$, $\alpha$. But, as will be clear from the proof, it can be chosen so as to remain valid for any $\alpha'$ in a neighbourhood of $\alpha$.
\end{rem}
%Note that we use 2.4 and 2.5 that require delta \to 1 to get uniform convergence, but if delta is fixed

\begin{proof} From Propositions~\ref{propexpdecrease},~\ref{propexpdecrease2}, the minimization problem \eqref{mep} admits a minimizer $U_\ep:\R\to\R^2$, and  the rescaled function $t\to U_\ep(\tep t)$ converges exponentially fast to its limits $a$ and $b$ as $t\to \pm\infty$, uniformly in $\ep$. Moreover, from Proposition~\ref{propexpdecrease}
$$\int_\R |1 - |U_\ep|^2| < C\ep.$$

Now let $\om_\alpha$ be a minimizer for \eqref{minper}. It is a domain with  analytic boundary  and we may define the signed distance function
\beq\label{lambda}\lambda_\alpha(x) = \dist(x,\om_\alpha) - \dist(x,\om_\alpha^c),\eeq
which is smooth in a neighbourhood of $\gamma_\alpha := D\cap \partial\om_\alpha$, say an $\eta$-neighbourhood, with bounds which are in fact independant of $\alpha$ in a neighbourhood of some, say, $\alpha_0\in (0,1)$ (to adress the above remark).

Now we modifiy the function $U_\ep$  as $\tUe$ so that $\tUe = a$ on $(-\infty,-\eta/\tep]$ and $\tUe = b$ on $[\eta/\ep+\infty)$. Because of the exponential convergence of $t\to U_\ep(\tep t)$ at infinity, this can be done in such a way that $\|\tUe- U_\ep\| < Ce^{- M/\tep}$, where $M>0$ and the norm is the $C^k$-norm for arbitrarily chosen $k$. It can also be done in such a way that
\beq\label{mod1}\int_\R |1 - |\tUe|^2| < C\ep\sqrt{\delta - 1}.\eeq

Then we let $v_\ep(x) = \tUe(t_\ep+\lambda_\alpha(x)/\tep)$, for some $t_\ep\in\R$. It is straightforward to check that there exists $C>0$ independent of $\ep$ such that, for a suitable choice of $t_\ep\in[-C,C]$, the map $v_\ep$ satisfies
\beq\label{contrainte1}\dashint_D |v_{1,\ep}|^2 = \alpha,\eeq
and that moreover $F_\ep(v_\ep) \le \mep\la + C$. Note that this last estimate could be improved to $F_\ep(v_\ep) \le \mep\la + C\tep$ by using the fact that $U_\ep$ is symmetric with respect to the origin and therefore that the curvature effect cancels to leading order on both sides of the interface.

It remains to modify $v_\ep$ in a way such that the second constraint in \eqref{mae} is satisfied. From \eqref{mod1}, \eqref{mae} we know that
\beq \label{contrainte2}\left|\dashint_D |v_{2,\ep}|^2  - (1-\alpha)\right| \le C\ep.\eeq
Then we modify $v_{2,\ep}$ as follows: we fix  $x$, $r$ depending only on $D$, $\alpha$ such that  $D(x,2r)\subset {\om_\alpha}^c$. Then, for $\ep$ small enough we have $v_{2,\ep} = 1$ on $D(x,r)$ since $v_\ep\sim b$ in an $\tep$-neighbourhood of $\gamma_\alpha$.

We let $\tilde v_{2,\ep}(y) = 1 + t (r - |y-x|)_+$ in $D(x,r)$ and $\tilde v_{2,\ep} = v_{2,\ep}$ elsewhere, for a suitably chosen $t\in\R$. From \eqref{contrainte2}, it follows that there exists $t\in (-C\ep, C\ep)$ such that
$$\dashint_D |\tilde v_{2,\ep}|^2   =  (1-\alpha).$$
We let $\tilde v_\ep = (v_{1,\ep}, \tilde v_{2,\ep})$. It is straightforward to check that,
$$F_\ep(\tilde v_\ep) \le F_\ep(v_\ep) + C \le \mep\la + C,$$
which proves the proposition.
\end{proof}

We deduce from the above the following lower bound for $\mae$.
\begin{coro} \label{mepmae} Assume $\alpha\in(0,1)$, $D$ is a smooth bounded domain, and let $\mae$, $\la$ be defined in \eqref{mae}, \eqref{minper}. There exists $C>0$ depending only on $D$ and $\alpha$ such that
\beq  \mep\la - C\ltep \le \mae.\eeq
\end{coro}
\begin{proof} Choose an arbitrary $\eta>0$ and apply Theorem~\ref{enloc} to a minimizer $(\rie,\rje)$ for \eqref{mae}, the minimum problem defining $\mae$. Then
from the estimate  \eqref{bmae}, we have $F_\ep(\rie, \rje) = \mae\le \mep\la + C$ and therefore \eqref{uppb} yields
$$\mep\la - C(C +\ltep) \le F_\ep(\riep, \rjep,V_\eta).$$
\end{proof} 

\section{Minimizers in the presence of rotation}

The minimization of $\El$ given by \eqref{jlambda} gives rise to a free boundary problem by using convex duality, and allows to define a first critical field as in the case of one component Bose-Einstein condensates and superconductors \cite{SSbook,IM1,Serf}.

\begin{prop} \label{propobs} Assume $\beta \ge 0$. Defining $\El$ as in \eqref{jlambda}, the minimizer $\Jl$ of $\El(\cdot,\omega)$ among divergence-free vector fields, where $\omega$ is a domain in $\R^2$, can be written $\Jl = \nabla^\perp h_\beta$, where $h_\beta$ is the unique minimizer for the problem
\beq\label{dualpb} \min\left\{\dm\int_\omega |\nabla h|^2 -  2\int_\omega h,\ \text{ $h = 0$ on $\partial \omega$ and $\|h\|_\infty \le 1/(2\beta)$}\right\}.\eeq
The function  $h_\beta$ is $C^{1,1}$ and defining $\mu_\beta := \curl \Jl +2$ we have $\mu_\beta = 2\chi_{\om_\beta}$, where $\chi_{\om_\beta}$ is the characteristic function of the set
$\{h_\beta = -1/(2\beta)\}.$ This set is understood to be empty if $\beta = 0$

Finally, $|\om_\beta| = 0$ (or equivalently $\mu_\beta = 0$) if and only if
\beq\label{hc1} \beta \le \beta_1 := \frac1{2\max |h_\omega|},\quad\text{where $\Delta h_\omega = - 2$ in $\omega$ and $h_\omega = 0$ on $\partial \omega$}.\eeq
\end{prop}

\begin{proof} Since we minimize among divergence-free vector fields, we may let $j = \nabla^\perp h$, and minimize
$$ I(h) = \dm\int_\omega |\nabla h|^2 + \Phi(h),\quad \Phi(h) = \frac1{2\beta}\int_\omega |\Delta h +2|,$$
with the understanding that $\Phi(h) = +\infty$ if $\Delta h +2$ is not a measure with finite total variation in $\omega$, or if $\beta = 0$ and $\Delta h +2$ is not equal to $0$. Then using standard results in convex analysis (see for instance \cite{brezis}) we know that
$$\inf_h I(h) = -\min_h J(h),\quad J(h) =  \dm\int_\omega |\nabla h|^2 + \Phi^*(-h).$$
Then we compute
\begin{equation*}
\begin{split}\Phi^*(h) &= \sup_k \int_\omega \nabla h\cdot\nabla k- \frac1{2\beta}\int_\omega |\Delta k +2|\\
&= \sup_k\left\{\int_{\partial \omega} h\partial_\nu k - \int_\omega h(\Delta k +2) - \frac1{2\beta}\int_\omega |\Delta k +2| + 2\int_\omega h\right\}.
\end{split}
\end{equation*}
It is not difficult to check that the supremum is equal to $+\infty$ if $h$ is not constant on $\partial \omega$, and we may take the constant to be zero because $\Phi^*(h+c) = \Phi^*(h)$ for any constant $c$. Then we easily  find that, assuming $h = 0$ on $\partial \omega$, the supremum is $+\infty$ if $\|h\|_\infty > 1/(2\beta)$, and that it is otherwise acheived when $\Delta k +2= 0$. Therefore
$$\Phi^*(h) = 2\int_\om h,\quad \min J(h) = \min_{\substack{h\in H^1_0(\om)\\\|h\|_\infty \le 1/(2\beta)}} \dm\int_\omega |\nabla h|^2 - 2\int_\omega h.$$
This proves the first part of the proposition, the rest being well known results on the obstacle problem, see \cite{cafa}, or \cite{Serf} for the last assertion.
\end{proof}

\subsection*{Upper bound, case A}

This follows closely the construction in \cite{SSbook}, Chapter~7, see also \cite{Serf} for an even more closely related construction, thus we will be a bit sketchy for the parts of the proof which can be found in these references.

We  assume that $\O/\ltep$ converges to $\beta\ge 0$ and we are going to construct a test couple $(\rie,\rje)$ such that
\beq\label{upcaseA} \EOep(\rie,\rje) \le \mep\la + \Omega^2 \(\min_{\diver j = 0} \El(j,\om_\alpha)+ \min_{\diver j = 0} \El(j,\om_\alpha^c)\) + o(\ltep^2),\eeq
where $\om_\alpha$ is a minimizer of \eqref{minper}.

We choose an arbitrary $\eta>0$, and let
\beq\label{vdelta}V_\eta = \{x\mid d(x,\gamma_\alpha) < \eta\}, \quad D_1 = \om_\alpha\cup V_\eta, \quad D_2 = \om_\alpha^c \cup V_\eta.\eeq We begin by defining the phase $\vpie$ (resp. $\vpje$) of $\rie$ (resp. $\rje$).
We need to define the phase $\vpie$ on $D_1$ rather than $\om_\alpha$ because the modulus of $\rie$ will not vanish outside $\om_\alpha$ exactly, but outside a slightly larger set. However $\eta$ is arbitrary and will be sent to $0$ eventually.

Denote by $h_1$ (resp. $h_2$) a minimizer of \eqref{dualpb} in $D_1$ (resp. $D_2$) and let $\mu_i = \Delta h_i +2$, $i = 1$, $2$. Then from Proposition~\ref{propobs} we have $\mu_i = 2\chi_{\om_{\alpha, i}}$, where $\chi_{\om_{\alpha, i}}$ is the characteristic function of $\om_{\alpha, i}$, defined as the set where $h_i$ is equal to $1/(2\beta)$, i.e. saturates the constraint in \eqref{dualpb}. Note that $\om_{\alpha,1}$ is a subset of $\om_\alpha$ while $\om_{\alpha,2}$ is a subset of $\om_\alpha^c$. We have
\beq \label{Ebeta} \min_{\diver j = 0} \El(j,D_i) = \El(\np h_i,D_i) = \dm\int_{D_i} |\nabla h_i|^2 + \frac1\beta|\om_{\alpha, i}|,\eeq
where it is understood in the case $\beta = 0$ that the second term is equal to $0$ since $\om_{\alpha, i} = \varnothing$ in this case.

{\em The simplest case is when}  $\beta < \min(\beta_1(\om_\alpha), \beta_1(\om_\alpha^c))$. Then by choosing $\eta$ small enough we have $\beta < \min(\beta_1(D_1), \beta_1(D_2))$ and thus $\om_{\alpha, 1}$ and $\om_{\alpha, 2}$ are empty, this is the case without vortices. Then we define
\beq\label{phasemeissner}\nabla \vpie = \O\nabla^\perp h_1 + \O x^\perp, \quad \nabla \vpje = \O\nabla^\perp h_2 + \O x^\perp.\eeq
Note that, since  $\Delta h_i +2 = 0$, $i = 1$, $2$, the right-hand sides above are curl-free hence they are indeed gradients of  well defined functions in $D_1$ (resp. $D_2$). Then we let $\rie = \vie e^{i\vpie}$, $\rje = \vje e^{i\vpje}$, where
$v_\ep = (\vie,\vje)$ is defined in  Proposition~\ref{propupscalar}. We have
\beq\label{boundrot1}\EOep(\rie,\rje) = F_\ep(\vie,\vje) +  \dm\int_{D_1}\viesq|\nabla \vpie - \O x^\perp|^2+  \dm\int_{D_2}\vjesq|\nabla \vpje - \O x^\perp|^2.\eeq
From Prop~\ref{propupscalar}, $F_\ep(\vie,\vje)$ is bounded above by $\la\mep +C$. Still from Proposition~\ref{propupscalar}, we have $|\vie|^2$, $|\vje|^2\le 1+C\ep$, where $C$ depends only on $\alpha$, $D$. Therefore, in view of \eqref{phasemeissner} we have
$$ \int_{D_1}\viesq|\nabla \vpie - \O x^\perp|^2+  \dm\int_{D_2}\vjesq|\nabla \vpje - \O x^\perp|^2\le \O^2 (1+C\ep) \(\int_{D_1} |\nabla h_1|^2 +\int_{D_2} |\nabla h_2|^2\). $$
Thus, in view of \eqref{Ebeta} and the fact that $\om_{\alpha, 1}$ and $\om_{\alpha, 2}$ are empty, we may write \eqref{boundrot1} as
$$\EOep(\rie,\rje) \le \la\mep +  (1+C\ep) \O^2 \(\min_{\diver j = 0} \El(j,D_1) + \min_{\diver j = 0} \El(j,D_2)\) +C,$$
which in turn implies that
$$\limsup_{\ep\to 0} \frac{\EOep(\rie,\rje)  -  \la\mep}{\O^2}\le \min_{\diver j = 0} \El(j,D_1) + \min_{\diver j = 0} \El(j,D_2).$$
This is not exactly \eqref{upcaseA} since the domain $D_1$ (resp. $D_2$) is not exactly equal to $\om_\alpha$ (resp. $\om_\alpha^c$). However \eqref{upcaseA} follows from the above when we let   $\eta\to 0$.\medskip

{\em The case where}  $\beta > \min(\beta_1(\om_\alpha), \beta_1(\om_\alpha^c))$, or equivalently the case where either $\om_{\alpha, 1}$ or $\om_{\alpha, 2}$ is nonempty is a bit more involved as it involves vortices.  As in \cite{SSbook}, Chapter~7, or \cite{Serf}, we may approximate $\O\mu_1$ (resp. $\O \mu_2$) by
$$\muie = 2\pi \sum_{i = 1}^{\nie} \muiie\quad\left(\text{resp.}\quad \muje = 2\pi \sum_{i = 1}^{\nje}\mujie\right),$$
where $\muiie$ (resp. $\mujie$) is the uniform positive measure of mass $2\pi$  in $B(\aie,\tep)$ (resp. $B(\bie,\tep)$) and $\{\aie\}_i$ (resp. $\{\bie\}_i$) are points in $\om_{\alpha, 1}$ (resp. $\om_{\alpha, 2}$) at distance at least $2\tep$ from one another chosen such that $\muie/\O$ (resp. $\muje$)  converges to $2\chi_{\om_{\alpha, 1}}$ (resp. $2\chi_{\om_{\alpha, 2}}$).

Then,  if we define $\hie\in H^1_0(D_1)$ (resp $\hje\in H^1_0(D_2)$) to satisfy
\beq\label{hie} \Delta \hie = \muie - 2\O\quad \left(\text{resp.}\quad \Delta \hje = \muje - 2\O\right),\eeq
it can be shown (see  \cite{SSbook} or \cite{Serf}) that, as $\ep\to 0$, for $k = 1,2$,
\beq\label{boundphase} \dm\int_{D_k}|\nabla h_{k,\ep}|^2 \le \O^2 \El(\nabla^\perp h_k,D_k) + o(\ltep^2).\eeq

Then we let
\beq\label{defphase} \nabla \vpie = \nabla^\perp \hie + \O x^\perp, \quad \nabla \vpje = \nabla^\perp \hje + \O x^\perp.\eeq
The fact that $\muiie$ (resp. $\mujie$) is a positive measure of mass $2\pi$ supported in $B(\aie,\tep)$ (resp. $B(\bie,\tep)$) and \eqref{hie} imply that this   indeed defines  gradients of  functions which are well defined modulo $2\pi$ in $D_1\setminus \cup_i B(\aie,\tep)$  and $D_2\setminus\cup_i B(\bie,\tep)$, respectively (see the aforementionned references for details.)
Note that \eqref{defphase} defines $\vpie$ (resp. $\vpje$)  only in $D_1$ (resp. $D_2$). Where it is not defined by \eqref{defphase}, we let the phases be $0$ which apriori induces a discontinuity, but in fact does not because the modulus of $\rie$ (resp. $\rje$) will be defined to be zero where the discontinuity occurs.

Now we define  the modulus of $\rie$ (rep. $\rje$). Define  $v_\ep = (\vie,\vje)$ as in  Proposition~\ref{propupscalar}. Recall that $\vie$ is equal to $1$ in ${D_2}^c$, and equal to $0$ in ${D_1}^c$ while $\vje$ is equal to $0$ in ${D_2}^c$ and $\vje-1$ is bounded by $C\ep$ in $C^1({D_1}^c)$.

We modify $v_\ep$ in the vortex balls: Let $\theta(r) = \pi/2$ if $r\in [0,1]$ and $\theta(r) = (2-r) \pi /2$ if $r\in (1,2)$. Let  $\rhoie = \vie$ and $\rhoje = \vje$ outside $\cup_i B(\aie,2\tep)\cup_i B(\bie,2\tep)$). For $x\in B(\aie,2\tep)$ let  $\rhoie(x) = \cos\theta(r/\tep)$, where $r = |x-\aie|$ and $\rhoje(x) = \sin\theta(r/\tep)$. For $x\in B(\bie,2\tep)$ let  $\rhoje(x) = \vje(x)\cos\theta(r/\tep)$ and $\rhoie = \sin\theta(r/\tep)$. Note that since the balls are centered at points belonging to either $\om_{\alpha, 1}$ or $\om_{\alpha, 2}$, they are at a fixed distance from the interface $\gamma_\alpha$ hence from (\ref{estvv}), $\vie$ is equal to either $0$ or $1$ on the balls while $\vje$ is either equal to $0$ or such that $\|\vje- 1\|_{C^1}<C\ep$. It is  straightforward to check that in any vortex ball $B = B(\aie,2\tep)$ or $B = B(\bie, 2\tep)$ we have  $F_\ep(\rhoie,\rhoje) < C$, where $C$ is independent of $\ep$. Therefore the total contribution of the balls to $F_\ep(\rhoie,\rhoje)$ is bounded by $C\O$.

We define  $\rie = \rhoie e^{i\vpie}$ and $\rje = \rhoje e^{i\vpje}$. Then \eqref{boundrot1} holds with $\rhoie$ replacing $\vie$ (resp. $\rhoje$ replacing $\vje$), and we deduce as above from Proposition~\ref{propupscalar} and \eqref{boundphase} that
\beq\label{almostbsup} \EOep(\rie,\rje)  \le \la\mep +  (1+C\ep) \O^2 \(\min_{\diver j = 0} \El(j,D_1) + \min_{\diver j = 0} \El(j,D_2)\) +C\O.\eeq
However we may not yet conclude that \eqref{upcaseA} is satisfied as in the previous case because $(\rie,\rje)$ does not satisfy the constraint \eqref{constraint}, due to   the modification of $(\vie, \vje)$ in the vortex balls. Since the number of balls is bounded by $C\ltep$ and their radius is $2\tep$, we have
\beq\label{firsterror} \left|\dashint_D |\rie|^2 - \alpha\right| \le  C\tep^2\ltep.\eeq
On the other hand
$$\left|\dashint_D |\rie|^2 + |\rje|^2 - 1\right| = \left|\dashint_{D}  |\rie|^2 + |\rje|^2 - \viesq - \vjesq\right|  =
\left|\dashint_{\cup_i B(\bie,2\tep)}  (1 - \vjesq) \sin^2\theta\right|,  $$ thus --- since $\vje-1$ is bounded by $C\ep$ in ${D_1}^c$ --- we deduce
\beq\label{seconderror}\left|\dashint_D |\rie|^2 + |\rje|^2 - 1 \right| \le C\ep.\eeq
To correct the first error we perturb the value of $\alpha$ relative to  which   $(\vie,\vje)$ is defined in the above construction. If $(\vie,\vje)$ is defined in  Proposition~\ref{propupscalar} with a value $\alpha+t$, and the definition of $\rie$ and $\rje$ is otherwise unchanged, then the average of $|\rie|^2$ over $D$ is a continuous function of $t$ and \eqref{firsterror} tells us that it is equal to $\alpha+t$ within an error $C\tep^2\ltep.$ Thus there exists $t_\ep$ such that $|t_\ep|\le C\tep^2\ltep$ and such that  the resulting  $(\rie,\rje)$ satisfies
$$\dashint_D |\rie|^2 = \alpha.$$
Then $|\rje|$ needs to be modified in order for the second constraint to be satisfied. In view of \eqref{seconderror}, this may be done as in the proof of Proposition~\ref{propupscalar} by adding to $\rje$ a correction which is bounded by $C\ep$ in $C^1(D)$. Still denoting $(\rie,\rje)$ the modified test configuration, the following modification of \eqref{almostbsup} holds:
$$\EOep(\rie,\rje)  \le \ell_{\alpha+t_\ep} \mep +  (1+C\ep) \O^2 \(\min_{\diver j = 0} \El(j,D_1) + \min_{\diver j = 0} \El(j,D_2)\) +C\O.$$
Since $\alpha\to\la$ is locally lipschitz, we have $\ell_{\alpha+t_\ep}\le \la + C|t_\ep|\le \la + C\tep^2\ltep$. Then grouping the error terms the above may be rewritten as
$$\EOep(\rie,\rje)  \le \la \mep + \O^2 \(\min_{\diver j = 0} \El(j,D_1) + \min_{\diver j = 0} \El(j,D_2)\) + C\O.$$
As in the case without vortices \eqref{upcaseA} follows by taking a suitable diagonal sequence $\ep\to 0$, $\eta\to 0$.

\subsection*{Upper bound, cases B and C} As above we choose an arbitrary $\eta>0$, and define $V_\eta$, $D_1$ and $D_2$ as in \eqref{vdelta}. We define a test configuration $(\rie = \rhoie e^{i\vpie},\rje = \rhoje e^{i\vpje})$  and then prove that
\beq\label{upcaseBC}\EOep(\rie,\rje) \le \frac{|D|}{2} \O \log\frac1{\tep\sqrt\O} + \mep\la + O(\O)\eeq

As in \cite{SSbook,Serf} we define the lattice
$$\Lambda_\ep = \sqrt{\frac\pi\O}\Z\times\sqrt{\frac\pi\O}\Z$$
and let $h_\ep$ be  the $\Lambda_\ep$-periodic solution of
$$ \Delta h_\ep = 2\pi\(\sum_{p\in\Lambda_\ep}\delta_p\) - 2\O$$
in $\R^2$. Then we let $\vphi_\ep$ be such that $\nabla \vphi_\ep = \np h_\ep + \O x^\perp$, so that $\vphi_\ep$ is well-defined modulo $2\pi$ outside $\Lambda_\ep$ since
$$\curl \nabla \vphi_\ep = \Delta h_\ep + 2\O = 2\pi\sum_{p\in\Lambda_\ep} \delta_p,$$
 As in \cite{SSbook, Serf}, it is straightforward to check that
\beq\label{enerphase}\dm\int_{D\sm\cup_{p\in\Lambda_\ep} B(p,\tep)} |\nabla\vphi_\ep - \O x^\perp|^2 = \dm\int_{D\sm\cup_{p\in\Lambda_\ep} B(p,\tep)} |\nabla h_\ep|^2 \le \frac{|D|}{2} \O \log\frac1{\tep\sqrt\O} + C\O.\eeq
Then we let
\beq\label{defvp} \vpie = \vphi_\ep \chi_{D_1},\quad \vpje = \vphi_\ep \chi_{D_2}.\eeq
Note that as above, the discontinuity in the phases $\vpie$, $\vpje$ is unimportant since the modulus will be zero where it occurs.

To define the modulus, let $\thetaper$ be periodic w.r.t. the square $[-1/(2\sqrt\O),1/(2\sqrt\O)]\times [-1/(2\sqrt\O),1/(2\sqrt\O)]$ and on this square let
$$\thetaper (x) = \begin{cases} \pi/2 & \text{if $|x| < \tep$,}\\ (2\tep - |x|) \pi/(2\tep) & \text{if $\tep\le|x| < 2\tep$,}\\ 0 & \text{otherwise.}\end{cases}$$
Then let $(\vie,\vje)$ be given by  Proposition~\ref{propupscalar} and define
\beq\label{defrho} \rhoie = \vie\cos\thetaper + \vje \sin\thetaper,\quad  \rhoje = -\vie\sin\thetaper + \vje \cos\thetaper,\eeq
so  that
\beq\label{fc}\dashint_D \rhoie^2+\rhoje^2 = \dashint_D \viesq+\vjesq = 1.\eeq
Also, since $\rhoie = \vie$ except on the balls of radius $2\tep$ centered on the lattice $\sqrt{\frac\pi\O}\Z\times\sqrt{\frac\pi\O}\Z$, we have
$$\left|\dashint_D |\rhoie|^2 - \alpha\right|  = \left|\dashint_D |\rhoie|^2 - \viesq\right| \le  C\tep^2\O.$$
As in the previous cases, there exists a real number  $t_\ep$ such that $|t_\ep| < C\tep^2\O$ and such that if we define $(\vie,\vje)$ by applying   Proposition~\ref{propupscalar} to $\alpha + t_\ep$ rather than $\alpha$ and $(\rhoie, \rhoje)$ by \eqref{defrho} then
$$\dashint_D |\rhoie|^2  = \alpha$$
and, using \eqref{fc},
$$\dashint_D |\rhoje|^2  = 1- \alpha.$$
Then let $\rie = \rhoie e^{i\vpie}$ and $\rje = \rhoje e^{i\vpje}$. From the previous considerations they satisfy the constraints in \eqref{thmin} and thus
$$\min \EOep \le \EOep(\rie,\rje),$$
which we estimate now.

First, since $\rhoie = 0$ outside $D_1$ and $\rhoje = 0$ outside $D_2$ we have
\begin{equation}\label{bsphase}\begin{split}
\dm\int_D \rhoie^2 |\nabla\vpie - \O x^\perp|^2 + \rhoje^2 |\nabla\vpje - \O x^\perp|^2&= \dm\int_D (\rhoie^2 + \rhoje^2) |\nabla h_\ep|^2 \\ &\le (1+C\ep) \(\frac{|D|}{2} \O \log\frac1{\tep\sqrt\O} + C\O\).\end{split}\end{equation}
To estimate the integral of $|\nabla\rhoie|^2$ ad $|\nabla\rhoje|^2$, we note first that from \eqref{defrho} we have
$$|\nabla \rhoie|^2 + |\nabla \rhoie|^2 =  |\nabla \vie|^2 + |\nabla \vje|^2 + |\nabla \thetaper|^2 (\viesq + \vjesq) + 2 \nabla \thetaper (\vje\nabla\vie - \vie\nabla\vje).$$
Then using the fact that $|\nabla\thetaper|$ is supported in $\cup_{p\in\Lambda_\ep} B(p,2\tep)$, bounded by $ C/\tep$, that $\vie$ and $\vje$ are bounded uniformly by $1+C\ep$, and that $\nabla\vie$, $\nabla\vje$ are bounded by $C/\tep$ we easily deduce that
\begin{equation}\label{bsrho}
\dm\int_D  |\nabla \rhoie|^2 + |\nabla \rhoie|^2 = \dm\int_D |\nabla \vie|^2 + |\nabla \vje|^2 + O\(\O\).  \end{equation}
It remains to estimate the integral of $\W(\rie,\rje)$ as defined in \eqref{vep}. From \eqref{defrho} we have
\beq\label{pot1}\int_D (1 - \rhoie^2 - \rhoje^2)^2 = \int_D ( 1 - \viesq - \vjesq )^2.\eeq
Moreover, outside $\cup_{p\in\Lambda_\ep} B(p,2\tep)$ we have $\rhoie^2\rhoje^2 = \viesq\vjesq$ therefore
\beq\label{pot2}\int_D \rhoie^2\rhoje^2 = \int_D \viesq\vjesq + O(\tep^2\O).\eeq
In view of \eqref{bsphase}, \eqref{bsrho}, \eqref{pot1}, \eqref{pot2} we deduce
$$\EOep(\rie,\rje) \le \frac{|D|}{2} \O \log\frac1{\tep\sqrt\O} + F_\ep(\vie,\vje) + O(\O)\le
\frac{|D|}{2} \O \log\frac1{\tep\sqrt\O} + \mep\la + O(\O),$$
proving \eqref{upcaseBC}.

\subsection*{Lower bound and convergence, Case A}

Assume $\ep>0$ and let $(\rie,\rje) $ be a minimizer  of $\EOep$. We let $\rhoie = |\rie|$, $\rhoje = |\rje|$. Then
\beq\label{decomp} \EOep(\rie,\rje) = F_\ep(\rhoie,\rhoje) + G_\ep(\rie,\rje), \eeq
where, defining $\Jie$, $\Jje$ as in \eqref{jeps},
\beq\label{split}G_\ep(\rie,\rje) =  \dm\int_{D} \frac{|\Jie|^2}{\rhoie^2}+   \frac{|\Jje|^2}{\rhoje^2}.\eeq
The term $F_\ep(\rhoie,\rhoje)$ contains  the terms in the energy which depend only on  the positive scalars $\rhoie$, $\rhoje$, and do not depend on the phases of $\rie$, $\rje$. From the definition \eqref{mae} of $\mae$ and Corollary~\ref{mepmae} we have
\beq\label{lowscal} F_\ep(\rhoie,\rhoje)\ge \mae \ge \mep\la - C\ltep.\eeq

On the other hand, assuming $\O = \beta\ltep$, we know from the upper-bound \eqref{upcaseA} proved above that
$\EOep(\rie,\rje) \le \mep\la + C\ltep^2$, which implies  that $F_\ep(\rhoie,\rhoje)\le \mep\la + C\ltep^2$. Then from Proposition~\ref{proplinfinitybound} and the bound of the energy by $C/\tep$, we have $\rie^2+\rje^2 - 1 < C\tep$, hence we may apply Theorem~\ref{enloc} to $(|\rie|,|\rje|)$ to find  that any sequence $\{\ep\}$ converging to $0$ admits a subsequence (not relabeled)  such that  $\rhoie\to\chi_\oma$ and $\rhoje\to\chi_{\oma^c}$ for some  minimizer $\oma$ of \eqref{minper}, and moreover that for any $\eta>0$ we have
\beq\label{locscal} F_\ep(\rhoie, \rhoje,V_\eta) \ge \mep\la - C \ltep^2,\eeq
where $V_\eta$ denotes an $\eta$-neighbourhood of $\gamma_\alpha :=\partial\oma\cap D$. Note that $\gamma_\alpha$ is smooth. It follows from \eqref{locscal} and \eqref{upcaseA}, in view of \eqref{decomp}, that
\beq\label{upaway} F_\ep(\rhoie, \rhoje,D\sm V_\eta) \le C \ltep^2,\quad G_\ep(\rie,\rje,D\sm V_\eta)\le C\ltep^2.\eeq

To obtain the desired lower-bound we will bound from below $G_\ep(\rie,\rje)$ on $D_1 = \oma\sm V_\eta$ and $D_2 = \oma^c\sm V_\eta$. For convenience, we choose $V_\eta$ such that $D_1$ and $D_2$ have smooth boundaries. The lower bound on each component will be that of a one-component condensate as computed in \cite{Serf}, see also \cite{ss1}, hence we will be a bit sketchy in the proof. From \eqref{upaway} we have
$$\int_{D_1} |\nabla \rhoie|^2+|\nabla \rhoje|^2 + \frac1{\ep^2} (1 - \rhoie^2 - \rhoje^2)^2 \le C\ltep^2.$$
Since $|\nabla(\rhoie^2+\rhoje^2)|^2\le C |\nabla \rhoie|^2+|\nabla \rhoje|^2$, it follows --- using the coarea formula as in \cite{SSbook}, Proposition~4.8, suitably adapted ---  that the set
$$ \{x\in D_1\mid |1 - \rhoie^2 - \rhoje^2| > \ltep^{-1}\}$$
may be  included in the union $A_\ep$ of a finite number of closed disjoint balls whose sum of radii is bounded by $C\ep\ltep^4$.

Similarly, still using \eqref{upaway}, we have
$$\int_{D_1} |\nabla \rhoie|^2+|\nabla \rhoje|^2 + \frac1{\tep^2} \rhoie^2 \rhoje^2  \le C\ltep^2,$$
from which we deduce using the fact that  $|\nabla(\rhoie^2\rhoje^2)|^2\le C |\nabla \rhoie|^2+|\nabla \rhoje|^2$, that the set
$$ \{x\in D_1\mid \rhoie^2 \rhoje^2 > \ltep^{-1}\}$$
may be  covered by  the union $B_\ep$ of a finite number of closed disjoint balls whose sum of radii is bounded by $C\tep\ltep^4$.

Let $D_1^\ep = D_1\sm(A_\ep\cup B_\ep).$ If $x\in D_1^\ep$ then both $\rhoie^2 \rhoje^2$ and $|1 - \rhoie^2 - \rhoje^2|$ are bounded by $C\ltep^{-1}$, from which we deduce, for $\tep>0$ small enough, that for each point, either $|1 - \rhoie^2| \le 4C\ltep^{-1}$ or $|1 - \rhoje^2| \le 4C\ltep^{-1}$. We prove that if $\ep>0$ is small enough, depending on $D_1$, then necessarily $|1 - \rhoie^2| \le 4C\ltep^{-1}$ holds. Indeed the number of connected components of $D_1$ and $D_1^\ep$ is the same if $\ep$ is small enough, this is because small closed disjoint balls are removed from $D_1$, and because the boundary of $D_1$ is smooth. Then, if $|1 - \rhoje^2| \le 4C\ltep^{-1}$ held for some $x\in D_1^\ep$, it would hold also on the corresponding connected component, which would contradict --- if $\tep$ is small enough --- the fact that the integral of $\rhoje^2$ on $D_1$ converges to $0$.

Thus we have $|1 - \rhoie^2| \le 4C\ltep^{-1}$ on $D_1^\ep$. From here, we may reproduce  the proof of the lower-bounds in \cite{SSbook}, Chapter~7 or \cite{Serf}, to deduce that $B_\ep$ may be included in a union of disjoint closed balls $B_1$,\dots,$B_k$ with total radius bounded by $C\ltep^{-10}$ (the power is chosen large enough but is not optimal) in such a way  that denoting by $d_i$ the winding number of $\rie$ on $\partial B_i$, with $d_i$ set to $0$ if $B_i$ intersects the complement of $D_1$, we have as $\ltep\to 0$,
\beq\label{lowballs}\int_{\cup_i B_i} \rhoie^2|\nabla \vpie - \O x^\perp|^2\ge \pi  \(\sum_{i=1}^k |d_i| \) \ltep (1 - o(1)).\eeq
Moreover, the estimate on the sum of the radii of the balls $B_i$ ensures (see \cite{SSbook} or \cite{Serf}) that, as $\ep \to 0$ and in the sense of distributions,
\beq\label{jac}\curl \Jie +2\O - \nue \to 0,\quad\text{where}\quad \nue = 2\pi\sum_i d_i\delta_{a_i},\eeq
and where  $a_i$ is the center of $B_i$.

 Now, \eqref{upaway} and the fact that $\rie^2+\rje^2 - 1 < C\tep$ imply that $\{\Jie/\O\}_\ep$ is bounded in $L^2(D_1)$, hence converges weakly in $L^2(D_1)$ to some $j_1$, modulo a subsequence. Moreover
 \beq\label{lowlim} \liminf_{\ep\to 0}  \frac{G_\ep(\Jie,\Jje,D_1)}{\O^2} \ge \dm \int_{D_1} |j_1|^2.\eeq

 From \eqref{lowballs} and \eqref{upaway} we deduce that $\{\nue/\O\}_\ep$ is bounded in the set of measures, hence again converges weakly modulo a subsequence. Therefore, using \eqref{jac}, $(\curl\Jie +2/\O)/\O$ converges in the sense of distributions to a measure $\mu_1$. Obviously we have $\mu_1 = \curl j_1+2$.

 In the case $\beta = 0$, the lower bound \eqref{lowballs} together with the apriori bound \eqref{upaway} implies that $\sum_i |d_i| \ll \O$ as $\ep\to 0$, hence $\mu_1 = 0$, using \eqref{jac}. Thus in this case $\curl j_1+2 = 0$ and we deduce directly from \eqref{lowlim} that
 $$ \liminf_{\ep\to 0}  \frac{G_\ep(\Jie,\Jje,D_1)}{\O^2} \ge \min_{\substack{ \diver j = 0\\\curl j +2 = 0}} \El(j,D_1),$$
 with a similar lower bound holding in  $D_2$ as well.

 When $\beta >0$, arguing as in  \cite{SSbook} or \cite{Serf}, since $\cup_i B_i$ has measure tending to $0$ as $\ep\to 0$, by going to a further subsequence we may add up the lower bounds \eqref{lowballs} and \eqref{lowlim} to find that
 $$\liminf_{\ep\to 0} \frac{G_\ep(\Jie,\Jje,D_1)}{\O^2} \ge \dm \int_{D_1} |j_1|^2 + \frac{1}{2\beta}\int_{D_1}|\curl j_1 +2|,$$
 where the last integral should be understood as the total variation of the measure $\mu_1$. The same argument in $D_2 = \oma^c\sm V_\eta$ yields
 $$\liminf_{\ep\to 0} \frac{G_\ep(\Jie,\Jje,D_2)}{\O^2} \ge \dm \int_{D_2} |j_2|^2 + \frac{1}{2\beta}\int_{D_1}|\curl j_2 +2|,$$
 where $j_2$ is the limit as $\ep\to 0$ of $\Jje/\O$.

 Adding the above lower bounds, either in the case $\beta = 0$ or $\beta>0$,  and in view of \eqref{decomp} we find that
 $$ \liminf_{\ep\to 0} \frac{\min \EOep - \mep\la}{\O^2} \ge  \min_{\diver j = 0} \El(j,D_1)+ \min_{\diver j = 0} \El(j,D_2).$$
 We recall that $D_1 = \oma\sm V_\eta$, $D_2 = \oma^c\sm V_\eta$. Since the above lower bound is true for any $\eta>0$, we deduce that the inequality holds with $\oma$ (resp. $\oma^c$) replacing $D_1$ (resp. $D_2$). This proves the lower part of \eqref{minA}.

 It is readily checked that for minimizers $(\rie,\rje)$, since the upper and lower bounds match, then necessarily $\Jie$ (resp. $\Jje$) converges modulo subsequences to a minimizer of $\El$ on $\oma$ (resp. $\oma^c$). This concludes the proof of Part~A of Theorem~\ref{thmin}.

\subsection*{Lower bound and convergence, Cases B and C}
The method to compute the lower bounds on the energy of minimizers in cases~B and~C is, as in \cite{ss2}, see also \cite{SSbook}, to suitably rescale things so that in rescaled coordinates the rotation $\O$ is not too large. Then a lower bound is computed along the lines of case~A on rescaled balls of radius one which correspond to small  balls in the original scale. The latter step is summarized in the following

\begin{lemm}\label{lowresc} Let $\EOep$ be as in \eqref{eqEnergyOmega} and $F_\ep$, $G_\ep$ be as in \eqref{decomp}, \eqref{split}. Assume that $\delta = \delta(\ep)$ and that $\tep\to 0,$ $\tep \gg\ep$ as $\ep\to 0$, where $\tep = \ep/\sqrt{\delta-1}$.

There exists $C>0$ such that for any $M>0$ the following holds:  if
\beq\label{ballom}\Omega = M\log\frac1{\tep\sqrt\O},\eeq
there exist $\ep_0>0$ such that for any $\ep<\ep_0$ and any $(\rie,\rje)$ defined on the unit ball $B$ such that  $F_\ep(|\rie|,|\rje|) < \ltep^4$,
\beq\label{lowsc} G_\ep(\Jie,\Jje,B) \ge \O |B| \ltep\(1 - C M^{-1/3}\),\eeq
where $\Jie$, $\Jje$ are defined in \eqref{jeps}.
\end{lemm}

We postpone the proof of this lemma to the end of this section.

We consider minimizers $\{(\rie,\rje)\}_\ep$ of $\EOep$ and define $\roie$, $\roje$ as in Case~A  and the currents $\Jie$, $\Jje$ as in \eqref{jeps}. We use the same splitting of the energy \eqref{decomp} as in case~A.

\subsubsection*{Case~B}

The upper bound \eqref{upcaseBC} and  Proposition~\ref{proplinfinitybound}  imply that, in Case~B, we have  $\rie^2+\rje^2 - 1 < C\tep$. Hence, following case~A, any sequence $\{\ep\}$ converging to $0$ admits a subsequence (not relabeled)  such that  $\rhoie\to\chi_\oma$ and $\rhoje\to\chi_{\oma^c}$ for some  minimizer $\oma$ of \eqref{minper} and  for any $\eta>0$ we have
\beq  F_\ep(\rhoie, \rhoje,V_\eta) \ge \mep\la - C  \O \log\frac1{\tep\sqrt\O},\eeq
where $V_\eta$ denotes an $\eta$-neighbourhood of $\gamma_\alpha :=\partial\oma\cap D$. It follows that
\beq\label{upawayB} F_\ep(\rhoie, \rhoje,D\sm V_\eta) \le C  \O \log\frac1{\tep\sqrt\O},\quad G_\ep(\rie,\rje,D\sm V_\eta)\le C \O \log\frac1{\tep\sqrt\O}.\eeq
The right-hand side in these bounds is negligible compared to $\O^2$ when $\ltep\ll\O$. Thus it follows that both $\Jie/\O$ and $\Jje/\O$  converge to $0$ on $D\sm V_\eta$ as $\ep\to 0$. Since this is true for arbitrary $\eta>0$, they converge to $0$ on $D$. The rest of this section is devoted to the proof of \eqref{minB}.

We change scales in order to apply Lemma~\ref{lowresc} on the new scale. Given $\pep = \lambda\ep$, we have for any open set $\om\subset D$
\beq\label{enscale} \EOep(\rie,\rje,\om) = \pE(\prie,\prje\pom),\eeq
where
$$ \pE(\prie,\prje,\pom) = \sum_{k=1}^2 \int_{\pom}\frac 12 |\nabla u_{k,\ep}'-i\pO x^{\perp} u_{k,\ep}'|^2 + \frac 1 {4\pep^2} (1-|\prie|^2-|\prje|^2)^2
+\frac {1}{2\ptep^2} |\prie|^2 |\prje|^2$$
and
\beq\label{scalerel} \prie(x) = \rie(\lambda x),\ \prje(x) = \rje(\lambda x),\ \pom = \lambda\om,\ \pep = \lambda\ep,\ \ptep = \lambda\tep,\ \pO = \O/\lambda^2.\eeq
Note that $\ptep\sqrt\pO = \tep\sqrt\O$.  For any $M$, we define  $\lamep$ to be such that
\beq \label{choicescale}
\frac\O{\lamep^2} = M\log\frac1{\tep\sqrt\O},\eeq
then
$$\pO = M\log\frac1{\ptep\sqrt\pO}.$$
In cases~B and~C of Theorem~\ref{thmin}, we have $\ltep\ll\O\ll1/\tep^2$, so that  $1\ll \lambda$ and $\tep \sqrt{\pO}\to 0$. Therefore,  $\ptep\ll 1$ as $\ep\to 0$ and $\pO\simeq M\lptep$.

If we define the recaled currents $\pjie$, $\pjje$ as in \eqref{jeps}, replacing there $\rie$, $\rje$, $\O$ by $\prie$, $\prje$, $\pO$, and if we  let $\proie = |\prie|$ and $\proje = |\prje|$,  then the rescaled energy splits in a similar fashion to \eqref{decomp}, \eqref{split} as
\beq\label{pdecomp} \pE(\prie,\prje) = F_\pep(\proie,\proje) + G_\pep(\rie,\rje), \eeq
where $F_\pep$, $G_\pep$ are defined as in \eqref{split}.

We are now ready to bound from below $G_\ep(\rie,\rje)$ on $D_1 = \oma\sm V_\eta$ and $D_2 = \oma^c\sm V_\eta$. Fom Fubini's Theorem and using the above rescaling we have
\beq\label{fubi} G_\ep(\rie,\rje, D_1) = \int_{x\in\R^2}\frac{\lamep^2}\pi G_\pep(\prie,\prje,B_{\lamep x}\cap D_1')\,dx,\eeq
where $B_{\lamep x}$ denotes the unit ball centered at $\lamep x$ and $D_1' = \lamep D_1$. A similar identity hold for $F_\ep$, and also when replacing $D_1$ with $D_2$. In particular,  using \eqref{upawayB}  and \eqref{scalerel}, we have
\beq\label{fubiscal} \frac\pi{\lamep^2} F_\ep(\rie,\rje, D_1) = \int_{x\in\R^2}F_\pep(\prie,\prje,B_{\lamep x}\cap D_1')\,dx \le C\pO\log\frac1{\ptep\sqrt\pO}.\eeq
Let  $$ A = \{x\in D_1\mid\text{$B(x,1/\lamep)\subset D_1)$ and $F_\pep(\proie,\proje, B_{\lamep x})\le \lptep^4$}\}.$$
From the definition of $A$ and \eqref{choicescale} we may apply Lemma~\ref{lowresc} for each $x\in A$ on the ball $B(\lamep x, 1)$ to the rescaled configuration $(\prie, \prje)$. Then inserting the lower-bound \eqref{lowsc} in \eqref{fubi} we find
$$ G_\ep(\rie,\rje, D_1) \ge |A| \frac{\lamep^2}\pi\pO |B| \lptep\(1 - C M^{-1/3}\) = |A| \O \lptep\(1 - C M^{-1/3}\).$$
Using the fact that  $1\ll\lamep$ and using \eqref{fubiscal} we deduce that $|A|\simeq |D_1|$ as $\ep\to 0$. Moreover,   the fact that $\pO\simeq M\lptep$,  implies that
$$\lptep\simeq \log \frac1{\ptep\sqrt\pO} = \log \frac1{\tep\sqrt\O}.$$
It follows that, as $\ep\to 0$,
$$ G_\ep(\rie,\rje, D_1) \ge  |D_1| \O \log \frac1{\tep\sqrt\O}\(1 - C M^{-1/3} - o(1)\).$$
Summing with the corresponding inequality on $D_2$, using the fact that $M$ can be chosen arbitrarily large, and as in case~A using the fact we can choose the size $\eta$ of the neighbourhood of the interface $V_\eta$ arbitrarily small, we deduce that, as $\ep\to 0$,
\beq\label{lowGB} G_\ep(\rie,\rje, D) \ge  |D| \O \log \frac1{\tep\sqrt\O}\(1 - o(1)\).\eeq
We add to the above the lower bound $F_\ep(\rie,\rje)\ge \mep\la - C\ltep$ which follows from Lemma~\ref{mepmae}, to obtain  the lower bound part of \eqref{minB}.

\subsubsection*{Case~C} Case~C  is  simpler than case~B. Using the same rescaling and using the same notation as above we have, using the fact that now the interface energy is negligible compared to $\O\log(1/\tep\sqrt\O)$,
$$\frac\pi{\lamep^2} F_\ep(\rie,\rje, D) = \int_{x\in\R^2}F_\pep(\prie,\prje,B_{\lamep x}\cap D)\,dx \le C\pO\log\frac1{\ptep\sqrt\pO}.$$
Then we let $A$ be the set of $x$ such that $B(x,1/\lamep)\subset D)$ and $F_\pep(\proie,\proje, B_{\lamep x})\le \lptep^4.$ As above $|A|\simeq |D|$ so that if we  apply Lemma~\ref{lowresc} for each $x\in A$ on the ball $B(\lamep x, 1)$ to the rescaled configuration $(\prie, \prje)$ we  find, as $\ep\to 0$,
$$G_\ep(\rie,\rje, D) \ge  |D| \O \log \frac1{\tep\O}\(1 - C M^{-1/3} - o(1)\).$$
Using the fact that $M$ can be chosen arbitrarily large and that $G_\ep\le \EOep$ we deduce that \eqref{minC} holds.

 \begin{proof}[Proof of Lemma~\ref{lowresc}]  We assume in this proof that
 \beq\label{aprio}G_\ep(\Jie,\Jje) \le \pi\O\ltep\le CM\ltep^2,\eeq
 otherwise there is nothing to prove. Here the second inequality is an easy consequence of \eqref{ballom}.

 To prove the Lemma, we first proceed as in case~A to construct vortex balls. Using the bound
 $$\int_B |\nabla \rhoie|^2+|\nabla \rhoje|^2 + \frac1{\tep^2} \rhoie^2 \rhoje^2  \le C\ltep^4,$$
  the set
 $$ \{x\in B\mid \text{ $\rhoie^2 \rhoje^2 > \ltep^{-1}\}$ or $|1 - \rhoie^2 - \rhoje^2\ge \ltep^{-1}$}\}$$
 may be  covered by  the union $A_\ep$ of a finite number of closed disjoint balls whose sum of radii is bounded by $C\tep\ltep^6$. Then $B_\ep = B\sm A_\ep$ is a connected set and such that either $|1 - \rhoie^2| \le C\ltep^{-1}$ or $|1 - \rhoje^2| \le C\ltep^{-1}$ on $B_\ep$. Without loss of generality, we assume that $|1 - \rhoie^2| \le C\ltep^{-1}$ on $B_\ep$.
 From here, the vortex-ball construction (see  \cite{SSbook}, Chapter~7 or \cite{Serf}) implies  that $A_\ep$ may be included in a union of disjoint closed balls $B_1$,\dots,$B_k$ with total radius bounded by $C\ltep^{-10}$  in such a way  that denoting by $d_i$ the winding number of $\rie$ on $\partial B_i$, with $d_i$ set to $0$ if $B_i$ intersects the complement of $B$, we have
 \beq\label{lowbb}\int_{\cup_i B_i} \frac{|\Jie|^2}{\rhoie^2}\ge \pi  \(\sum_{i=1}^k |d_i| \)\( \ltep - C\log\ltep\).\eeq
 Moreover, the estimate on the sum of the radii of the balls $B_i$ ensures (see \cite{SSbook}, chapter~6 or \cite{Serf}) that, as $\ep \to 0$ and in the sense of distributions,
 \beq\label{jacb}\left\|\curl \Jie +2\O - \nue\right\|_{(C^{0,1})^*}\le C \ltep^{-6},\quad\text{where}\quad \nue = 2\pi\sum_i d_i\delta_{a_i},\eeq
 and where  $a_i$ is the center of $B_i$.

  Next we use \eqref{jacb} to estimate the sum of degrees in \eqref{lowb}, which will yield the desired result. Let $0\le\zeta\le 1$ be a function equal to $1$ on the ball of radius $1 - M^{-1/3}$, equal to $0$ on $\partial B$, and such that $|\nabla \zeta| \le M^{1/3}$. Then we have, using \eqref{aprio}, that
  $$\left|\int_B \zeta\curl \Jie \right| = \left|\int_B \nabla^\perp\zeta\cdot\Jie\right|\le CM^{1/6} \|\Jie\|_{L^2(B)} \le CM^{2/3}\ltep.$$
  Then, from \eqref{jacb},
  $$ \left|\int_B \zeta\(\curl \Jie +2\O - \nue\)\right| \le C\ltep^{-5}.$$
  We deduce that
  $$2\pi \sum_i d_i\zeta(a_i) \ge 2\O \int_B\zeta -  C\ltep^{-5} - CM^{2/3}\ltep,$$
  and then using the fact that $\O \simeq M\ltep$,  that when $\ep$ is small enough we have
  $$2\pi \sum_i d_i\zeta(a_i) \ge 2\O |B| (1 - CM^{-1/3}).$$
  Inserting in \eqref{lowbb} yields the desired result.
  \end{proof}

\section{Localisation of the line energy}

 We recall the definition of the energy $F_\ep$ of a pair $\ri, \rj:D\to\R$ by \eqref{scalar_energy}, with \eqref{vep1}
  where it is understood that $\delta$ is a function of $\ep$. We recall the definition of \eqref{mae}.

\subsection{Localisation of perimeter}

We start with the following quantitative convergence result for the perimeter.

\begin{prop} \label{locper} Let $D$ be a bounded smooth domain in $\R^2$ and $\alpha\in (0,1)$. Then for any $\eta>0$ there exists $C>0$ such that if $\om\subset D$ is such that
$ |\om| = \alpha |D|$, then there exists a minimizer $\om_\alpha$ of \eqref{minper} such that
\beq\label{boundper} \per_{D\setminus V_\eta}(\om) \le C \(\per_D(\om) - \la\),\eeq
where we denoted by $V_\eta$ a $\eta$-neighbourhood of the curve $\gamma_\alpha = D\cap \partial\om_\alpha$.
\end{prop}

\begin{proof}
We will prove the equivalent statement that, under the hypothesis of the proposition and given $\eta>0$,  if   $\{\om_n\}_n$ is a minimizing sequence for \eqref{minper}  then there exists a subsequence $\{n'\}$,  a minimizer $\om_\alpha$ of \eqref{minper}, and $C>0$  such that  if $n'$ is large enough (depending on $\eta>0$) then
$$ \per_{D\setminus V_\eta}(\om_{n'}) \le C \(\per_{D}(\om_{n'}) - \la\).$$

It is well known that if $\{\om_n\}_n$ is a minimizing sequence for \eqref{minper}, then there exists a subsequence $\{n'\}$ such that $\{\chi_{\om_{n'}}\}_{n'}$ converges weakly in $BV$ and strongly in $L^1$ to $\chi_{\om_\alpha}$, where $\om_\alpha$ is a minimizer \cite{giusti}. From now on we label $\{n\}$ the subsequence to lighten notation.

Let $V_t = \{x\in D\mid d(x,\om_\alpha) < t\}$  and $W_t = \{x\in D\mid d(x,\om_\alpha^c) < t\}$. By taking $\eta$ smaller if necessary, we may assume that  there exists a positive lower bound for the length of the connected components of $\partial V_t\cap D$ and  $\partial W_t\cap D$ for each $t\in (0,\eta)$.

Now, using the above convergence,
$$\lim_{n\to +\infty} \int_{V_\eta\sm V_{\eta/2}} |\nabla \chi_{\om_n}| = 0.$$
Moreover, let  $\gamma_n :=\partial \om_n\cap D$, using the coarea formula we have, ,
$$ \int_{V_\eta\sm V_{\eta/2}} |\nabla \chi_{\om_n}| \ge \int_{\eta/2}^\eta \#(\gamma_n\cap \partial V_t)\,dt,$$
where $\#A$ is the cardinal of $A$. It follows from the above and a mean value argument that for $n$ large enough, there exists $t_n\in(\eta/2, \eta)$ such that $\gamma_n\cap \partial V_n = \varnothing$, where we wrote $V_n$ for $V_{t_n}$.
Moreover,  from the $L^1$ convergence of $\chi_{\om_n}$ to $\chi_{\om_\alpha}$ we may assume --- by using again a mean value argument to determine $t_n$ ---  that
\beq\label{ccbord}\lim_{n\to +\infty} \ell(\om_\alpha\cap \partial V_n) = 0.\eeq

From $\gamma_n\cap \partial V_n = \varnothing$ we deduce that if $n$ is large enough, then each connected component of $D\cap \partial V_n$ is either included  in $\om_n$ or in $\om_n^c$. The former is not possible if $n$ is large because it would imply a lower-bound for $\ell(\om_\alpha\cap \partial V_n)$ contradicting \eqref{ccbord}. Therefore
$$ \partial V_n \subset \om_n^c.$$

Now let $A_n = \om_n\cap V_n$, $B_n = \om_n\cap V_n^c$. Then,
\beq\label{setup} \alpha |D| = |A_n| + |B_n|, \quad \ell(\gamma_n\cap V_n^c) = \ell(\partial B_n\cap D),\eeq
while $|B_n|\to 0$ as $n\to +\infty$ from the $L^1$ convergence.

We now use two well-known facts about isoperimetric problems in two dimensions (see for instance \cite{ros}). First, the function $\alpha\to\la$ is locally lipschitz on the interval $(0,1)$ and second, there exists a constant $C$ depending on the smooth domain $D$ such that for any $\om\subset D$ we have $|\om|\le C \ell(\om\cap D)^2.$ In view of \eqref{setup} we deduce that
\begin{align*}
\la - C \ell(\gamma_n\cap V_n^c)^2 &= \la - C \ell(\partial B_n\cap D)^2  \le \la - C \frac{|B_n|}{|D|} \\
 & \le \ell_{|A_n|/|D|} \le \ell(\partial A_n\cap D) = \ell(\gamma_n\cap V_n) +\ell(\om_n\cap\partial V_n)\\
 & = \ell(\gamma_n\cap V_n) = \ell(\gamma_n) - \ell(\gamma_n\cap V_n^c).
 \end{align*}
We deduce that
$$ \ell(\gamma_n\cap V_n^c)  - C \ell(\gamma_n\cap V_n^c)^2 \le \ell(\gamma_n) - \ell(\gamma_\alpha),$$
and then, using the fact that $\ell(\gamma_n) - \la$ tends to $0$ as $n\to +\infty$, that for $n$ large enough
$$ \ell(\gamma_n\cap V_n^c)  \le C\( \ell(\gamma_n) - \ell(\gamma_\alpha)\),$$
where if fact the constant can be taken as close to $1$ as one wishes.

We may similarly find $t_n\in(\eta/2,\eta)$ such that, letting $W_n = W_{t_n}$ we have, for $n$ large enough,
$$ \ell(\gamma_n\cap W_n^c)  \le C\( \ell(\gamma_n) - \ell(\gamma_\alpha)\).$$

It follows then from  $V_\eta^c\subset V_{n}^c\cup W_{n}^c$ that
$$\ell(\gamma_n\cap V_\eta^c) \le \ell(\gamma_n\cap V_n^c) + \ell(\gamma_n\cap W_n^c) \le C\( \ell(\gamma_n) - \ell(\gamma_\alpha)\),$$
proving \eqref{boundper} and  the proposition.
\end{proof}

\subsection{Lower bound from perimeter}
 Here we restate a result of P.Sternberg \cite{sternberg}.

In what follows we are given a two-well potential $W:\R^n\to \R$, where $n$ is a positive integer, such that $W$ is, say, $C^2$, nonnegative, and vanishes at exactly two points $a$ and $b$ where we assume moreover that the hessian of $W$ is positive definite.

We define for any $x\in\R^n$
\beq d(x,a) = \inf\left\{\int_{-\infty}^0 \dm|\gamma'(t)|^2 + W(\gamma(t))\,dt\mid \gamma:\R_-\to\R^n,\ \lim_{-\infty}\gamma = a,\ \gamma(0) = x\right\}.\label{da}\eeq
$$ d(x,b) = \inf\left\{\int_0^{+\infty} \dm|\gamma'(t)|^2 + W(\gamma(t))\,dt\mid \gamma:\R_+\to\R^n,\ \lim_{+\infty}\gamma = b,\ \gamma(0) = x\right\}.$$
$$ d(a,b) = \inf\left\{\int_{-\infty}^{+\infty} \dm|\gamma'(t)|^2 + W(\gamma(t))\,dt\mid \gamma:\R_+\to\R^n,\ \lim_{-\infty}\gamma = a,\ \lim_{+\infty}\gamma = b\right\}.$$
and we let
\beq\label{dv} d(x) = \begin{cases} d(x,a) & \text{if $d(x,a)< d(a,b)/2$,} \\
                                                          d(a,b) - d(x,b) & \text{if $d(x,b)< d(a,b)/2$,} \\
                                                          d(a,b)/2 & \text{otherwise.}
                                     \end{cases}
\eeq
Note that with  this notation $d(b)$ is the same as $d(a,b)$.

\begin{prop} (\cite{sternberg}) Given a smooth bounded domain $D\subset \R^2$, for any $u:D\to\R^n$ we have
$$F(u) := \int_D \dm|\nabla u(x)|^2 + W(u(x))\,dx \ge \int_0^{d(a,b)} \per_D\(\left\{x\in D\mid d(u(x)) < t\right\}\)\,dt.$$
\end{prop}

We include a sketch of the proof for the convenience o the reader.

\begin{lemm} The function  $x\to d(x)$ is locally lipschitz on $\R^n$ and $d(x)\in [0,d(a,b)]$ for any $x$. Morever  $|\nabla d(x)|  =  \sqrt{2 W(x)}$
a.e. on the set of $x$ such that $d(x) \neq d(a,b)/2$.
\end{lemm}
\begin{proof}[Proof (Guy Barles, oral communication).] Given $x\in\R^n$ such that $d(x,a)< d(a,b)/2$ it is not difficult to prove by the direct method that the infimum defining $d(x,a) $ is acheived by a certain $\gamma$. Then, given $h\in\R^n$, we extend $\gamma$ by letting
$$\gamma(t) = x + \sqrt{2 W(x)} t \frac{h}{|h|},\quad 0\le t\le \frac{|h|}{\sqrt{2 W(x)}}.$$
then $\gamma$ connects $a$ to $x+h$ and using it as a test path in the definition of $d(x,a)$ we find that
\beq \label{pathbound} d(x+h,a) \le d(x,a) + \sqrt{2 W(x)} h + o(|h|).\eeq
This shows that $x\to d(x,a)$ is locally lipschitz and, using a similar argument, we fins that $x\to d(x,b)$ is lipschitz as well, which then implies that $x\to d(x)$ is locally lipschitz too.

Using Rademacher's theorem, the function $d$ is then differentiable almost everywhere and \eqref{pathbound} together with its equivalent for $d(x+h,b)$  shows that $|\nabla d(x)|  \le  \sqrt{2 W(x)}$ at any $x$ where $d$ is differentiable. To prove the converse inequality we consider again some $x\in\R^n$ such that $d(x,a) < d(a,b)/2$ and a minimizer $\gamma$. Then $\gamma$ is smooth and satisfies $\gamma'' = \nabla W(\gamma)$ and thus
$$ \(\dm|\gamma'(t)|^2 - W(\gamma(t)\)' = 0.$$
Because $\gamma(t)\to a$ and $\gamma'(t)\to 0$ as $t\to -\infty$ we deduce that $|\gamma'(t)| = \sqrt{2W(\gamma(t))}$ for every $t$. Then the path  $t\to\gamma(t+\tau)$  defined on $(-\infty,0]$ is a minimizer for $d(\gamma(-\tau),a)$ hence
$$d(\gamma(-\tau), a) =  d(x) - \tau \left(\dm |\gamma'(0)|^2 + W(\gamma(0))\) + o(\tau) = d(x) - \tau \sqrt{2W(x)} |\gamma'(0)| +o(\tau),$$
which implies that $|\nabla d(x)|  \ge  \sqrt{2 W(x)}$ if $d$ is differentiable at $x$.
\end{proof}

Then the proof of the proposition follows the classical argument of Modica-Mortola \cite{momo} using the coarea formula.

First we use the well nown fact that $|\nabla d(x)|  = 0$ almost everywhere on $D_0 := \{x\in D\mid d(x) = d(a,b)/2\}$ --- this is true of any sobolev function on any level set, the catch being that generically the level set itself is negligible. Therefore we have
$$ F(u) \ge 2\int_D \frac{|\nabla u|}{\sqrt 2} \sqrt{W(u)} \ge \int_{D\setminus D_0} |\nabla(d\circ u)| = \int_{D} |\nabla(d\circ u)|.$$
Using the coarea formula, we deduce that
$$F(u) \ge \int_0^{d(a,b)}  \per_{D}\{d\circ u = t\}\,dt.$$

\subsection{Specialization to two-component condensates}

We now specialize the preceding section to the potential $\W$ defined in \eqref{vep}
so that the functional \eqref{scalar_energy} may be rewritten as
$$F_\ep(u) =\int_D \dm|\nabla u(x)|^2 + \W(u(x))\,dx.$$
The potential $\W$ is a two-well potential with wells $a = (1,0)$ and $b = (0,1)$.

We will denote by $\dep$ the distance function associated to the potential $\W$ by \eqref{dv} and let $\mep = \dep(a,b)$ so that $\mep$ is given by \eqref{mep}.
From \cite{AAJRL2013} it follows that $\mep \sim m_0/\tep$ for some positive constant $m_0$.

To study the behaviour of $\dep(x)$ when $x$ is close to the wells $a$ and $b$, we use polar coordinates $(r_x,\theta_x)$ in $\R^2$. Then, $x$ close to $a$ corresponds to $r_x$ close to $1$ and $\theta_x$ close to $0$, while $x$ close to $b$ corresponds to $r_x$ being close to $1$ and $\theta_x$ close to $\pi/2$. We have from \eqref{da} that
\beq\label{dapol}\dep(x,a) = \inf\left\{\int_{-\infty}^0 \dm ({r'}^2 + r^2{\theta'}^2) + \frac 1 {4\ep^2} (1-r^2)^2 +\frac {1}{2\tep^2} r^4\cos^2\theta\sin^2\theta\right\}, \eeq
where the infimum is taken over all functions $r$, $\theta$ such that $r(-\infty) = 1$, $\theta(-\infty) = 0$, $r(0) = r_x$ and $\theta(0) = \theta_x$.

If we only take the infimum with respect to the function $r$, keeping $\theta$ fixed equal to the constant $0$ the minimizer is $r(t) = \tanh(C - \frac t{\sqrt 2\ep})$ or $r(t) = \coth(C - \frac t{\sqrt 2\ep})$  according to wether $r_x\le 1$ or $r_x\ge 1$ and the minimum is
\beq\label{dep1}\depr(r_x) = \frac{(1-r_x)^2)(r_x+2)}{3\ep\sqrt 2}.\eeq
If on the other hand we fix $r$ to be the constant $1$ and minimize with respect to $\theta$ we find that the minimizer is $\theta (t) = \arctan(C+\frac x\ep),$
while the minimum is
\beq\label{dep2}\dept(\theta_x) =  \frac{\sin^2\theta_x}\tep.\eeq
It is straightforward to check that $$\dep(x,a)\le\depr(r_x) +\dept(\theta_x).$$ Moreover, if we know that the minimizer $(r,\theta)$ in \eqref{dapol} satisfies $\min r = r_{\min}$ then we have
\beq\label{depbounds}\depr(r_{\min}) +{r_{\min}}^2\dept(\theta_x)\le \dep(x,a).\eeq

From these fact, we deduce the following useful behaviour of $\dep$ near $a$ and $b$.

\begin{lemm}\label{Wd} There exist $\eta,C>0$ such that for any $\ep$ small enough and  any $x=(x_1,x_2)\in \R_+\times\R_+$ we have, letting   $\tep := \ep/{\sqrt{\delta - 1}}$:

\begin{enumerate}
\item There holds
$$ \min(\dep(x), \mep - \dep(x))\le C \tep\, \W(x).$$

\item If $ \dep(x)< \frac\eta\tep$ then $\frac{{x_2}^2}{C\tep} \le \dep(x)$

\item  If $ \dep(x) > \mep -  \frac\eta\tep$ then $\frac{{x_1}^2}{C\tep} \le \mep - \dep(x).$
\end{enumerate}
\end{lemm}

\begin{proof} We begin by proving item~2. First we claim that if $\eta$ is chosen small enough, then $ \dep(x)< \frac\eta\tep$ implies that $\dep(x) < \dep(a,b)/2$ hence $\dep(x) = \dep(x,a)$. Indeed if we consider a minimizing path for $\dep(a,b)$ and let $x_0$ denote its midpoint, then from symmetry considerations we have that $\theta_0 = \pi/4$, where $(r_0,\theta_0)$ denote the polar coordinates of $x_0$, and $\dep(x_0,a) = \dep(a,b)/2$. Then, from \eqref{depbounds} and the apriori bound $\dep(a,b)\le \dept(\pi/2) = 1/\tep$ we deduce, since $\ep\ll\tep$ as $\ep\to 0$, that $r_{\min}>1/2$ if $\ep$ is chosen small enough. It follows, still using \eqref{depbounds}, that $1/(8\tep)\le \dep(x_0,a).$ This proves the claim, with $\eta = 1/8$.

Now assuming $ \dep(x)< \frac\eta\tep$ and considering a minimizing path for $\dep(a,x)$, we deduce as above from \eqref{dep1}, \eqref{depbounds} that if $\ep$ small enough then  $r_{\min}>1/2$. Plugging this information in \eqref{depbounds} we find that $\sin^2\theta_x \le 4\eta$, which implies that $\theta_x < \pi/4$ if $\eta$ is chosen small enough. Then $x_2\le C\theta_x$ for a suitable $C>0$ and $1/\sqrt 2 < \cos(\theta_x)$. We deduce that, with a possibly different constant $C$,
$$\frac{{x_2}^2}{\tep} \le C \frac{{\sin}^2\theta_x}{\tep} = C \dept(\theta_x)\le  4 C\dep(x).$$
This proves item~2 of the lemma. Since the proof of  item~3 is very similar, we omit it.

Item 1  easily follows from the bound $\dep(x,a)\le\depr(r_x) +\dept(\theta_x)$ --- and a similar inequality for $\dep(x,b)$ ---  and \eqref{dep1}, \eqref{dep2}, using the fact that  $\ep\ll\tep$.
\end{proof}

\subsection{Area of level-sets}

We need the following:

\begin{lemm} \label{lemmarea}Assume that $\alpha\in (0,1)$, that $D$ is a bounded smooth domain in $\R^2$ and that  $C_0>0$ is an arbitrary constant. Then there exist $\ep_0, C>0$ such that the following holds.

For any $\ep\in(0,\ep_0)$ and any locally lipschitz  $u:D\to \R^2$ such that
\beq\label{hyps} F_\ep(u) \le \mep\la + \Dep,\quad \dashint_D \ri^2  = \alpha,\quad \rie^2+\rje^2 - 1 \le  C_0\tep,\eeq
it holds that, for any $t\in (C\ltep, m_\ep - C\ltep)$,
\beq\label{concprop} |\om_t - \alpha |D|| \le C\tep (\ltep+\Dep),\eeq
where we used the notation
\beq\label{omt} \om_t = \left|\{x\in D\mid \dep(u(x)) \le t\}\right|.\eeq
\end{lemm}

\begin{proof}
Since $t\to |\omega_t|$ increases continuously from $0$ to $|D|$, it suffices to prove first that, for some $C>0$,
\beq\label{smallarea}\left|\om_{\mep - C\ltep}\setminus \om_{C\ltep}\right|\le C\tep (\ltep+\Dep),\eeq
and second that, choosing $\widehat C>0$ large enough depending on $C$,
\beq \label{texists} \exists t\in\left[C\ltep, \mep - C\ltep\right],\quad \left| |\om_t| - \alpha |D|\right| \le \widehat C\tep \ltep.\eeq

We begin by proving the second claim, namely that given $C>0$ we may choose $\widehat  C>0$ such that \eqref{texists} holds. For this we assume that \eqref{texists} is not true for some $\widehat C>0$ and prove that $\widehat C$ cannot be too large. Since $t\to|\om_t|$ is increasing, the fact that \eqref{texists} is false implies that either
\beq\label{absurd}\forall t\le \mep - C\ltep,\quad |\om_t| \le  \alpha |D| - \widehat C\tep \ltep,\eeq
or for every $t\ge C\ltep$ we have $|\om_t| \ge \alpha |D| + \widehat C\tep \ltep$. We will assume the former, the other case can be treated in a similar fashion.

We have \beq\label{splitmass}\alpha |D| = \int_{\om_{\mep - C\ltep}} \ri^2 + \int_{D\sm \om_{\mep - C\ltep}} \ri^2.\eeq
Using item~3 of the previous lemma, we know that $\ri^2 \le C_1 C\tep \ltep$ on $D\sm \om_{\mep - C\ltep}$, where $C_1$ is the constant occuring in Lemma~\ref{Wd}. Moreover, using \eqref{hyps},
$$\int_{\om_{\mep - C\ltep}} \ri^2 =  |\om_{\mep - C\ltep}| + \int_{\om_{\mep - C\ltep}} \(\ri^2 - 1\)\le  |\om_{\mep - C\ltep}| + C_0\tep.$$
Then we deduce from \eqref{splitmass} and \eqref{absurd} that
$$ \alpha |D| \le \alpha |D| - \widehat C\tep \ltep + C_1 C\tep \ltep + C_0\tep,$$
which is clearly a contradiction  if $\widehat C$ is large enough and $\tep$ is small enough, thus proving \eqref{texists}.

It remains to prove \eqref{smallarea}, which is the crucial point  in the proof of Theorem~\ref{enloc}.  First we introduce some notation: Since $u$ is locally lipschitz, we know that for almost every $t>0$ the level set $\{\dep\circ u = t\}$ is empty or a lipschitz curve and we may define
$$\gamma_t = \{x\in D\mid \dep(u(x)) = t\},\quad  v(t) = \dashint_{\gamma_t} \frac{2\W(u(x))}{|\nabla(\dep\circ u) (y)|}\,d\ell(y),\quad  a(t) = \int_{\gamma_t} \frac{d\ell(y)}{|\nabla(\dep\circ u) (y)|},$$
where $d\ell$ denotes the line element on the curve $\gamma_t$.

We have, using the coarea formula, and letting $I_\ep = [C\ltep, \mep - C\ltep]$,
$$\left|\om_{\mep - C\ltep}\setminus \om_{C\ltep}\right| = \int_{t\in I_\ep} a(t)\,dt.$$
From Lemma~\ref{Wd} we know that $\W(u)\ge (C\tep)^{-1} \min(\dep(u),\mep-\dep(u))$. Therefore
$$  a(t) \le \dm \frac\tep t|\gamma_t| v(t),$$
where $|\gamma_t|$ denotes the length of the curve $\gamma_t$.
It follows that
\beq\label{diffarea} \left|\om_{\mep - C\ltep}\setminus \om_{C\ltep}\right| \le  \dm\int_{t\in I_\ep}  \frac\tep t|\gamma_t| v(t)\,dt\eeq

On the other hand,  using again the coarea formula,
\begin{equation*}
\begin{split}
F_\ep(u) &= \dm\int_{D} |\nabla u|^2 + \int_{D} \W(u)\\
 {} &\ge \int_0^\mep \int_{\gamma_t} \dm \frac{|\nabla u|^2}{|\nabla(\dep\circ u)|} +  \frac{\W(u)}{|\nabla(\dep\circ u)|}\,d\ell\,dt.
\end{split}
\end{equation*}
Using Jensen's inequality and the fact that $|\nabla \dep\circ u|\le |\nabla \dep(u)||\nabla u| = \sqrt{2\W(u)}|\nabla u|$, we have
$$ \dashint \frac{|\nabla u|^2}{|\nabla(\dep\circ u)|} \ge \(\dashint \frac{|\nabla(\dep\circ u)|}{|\nabla u|^2}\)^{-1} \ge \frac1{v(t)}.$$
It follows that
$$F_\ep(u) \ge \int_0^\mep \frac{|\gamma_t|}{2}\( v(t) + \frac1{v(t)}\)\,dt.$$
We may then substract $\mep\la$ and obtain, in view of our hypothesis \eqref{hyps}
\beq \label{enbup}\begin{split} \Dep & \ge F_\ep(u) - \mep\la \\
                                                     & \ge \int_0^\mep \frac{|\gamma_t|}{2}\( v(t) + \frac1{v(t)}\)\,dt - \mep\la - C\tep\\
                                                     & \ge \int_{t\in I_\ep} \frac{|\gamma_t|}{2}\( v(t) + \frac1{v(t)}\) - \la \,dt - C\ltep.
                                                     \end{split}
\eeq
Let $\delta(t) = \frac{|\gamma_t|}{2}\( v(t) + \frac1{v(t)}\) - \la$. We wish to bound from above the integrand in \eqref{diffarea}, possibly in terms of $\delta(t)$. We distinguish several cases, $C$ denotes a generic constant independant of $\ep$.

\begin{itemize}
\item[-] If $\la\le |\gamma_t| v(t) /  4$ then $\delta(t)\ge |\gamma_t|v(t) /2$ and therefore, using the fact that $t\ge C\ltep$, if $\ep$ is small enough then
$$ \frac\tep t|\gamma_t| v(t)\le v(t)\le C\tep \delta(t).$$
\item[-] If $|\gamma_t| v(t) \le 4\la $ then, since $\la$ is independent of $\ep$,
$$\frac\tep t|\gamma_t| v(t)\le C\frac{\tep}{t}.$$
\end{itemize}
It follows, in view of \eqref{enbup} and since $I_\ep = [C\ltep, \mep - C\ltep]$, that
\beq\label{almostarea}\left|\om_{\mep - C\ltep}\setminus \om_{C\ltep}\right| \le C\tep \int_{t\in I_\ep}  \frac1t+ \delta_+(t)\,dt\le C\tep\(\ltep + \Dep + \int_{t\in I_\ep} \delta_-(t)\,dt\).\eeq
It remains to bound the last integral on the right-hand side. For this we note that, since $\delta(t) \ge |\gamma(t)| - \la$, we have
$$\delta_-(t) \le \(\la - |\gamma(t)|\)_+.$$
But  $\la - |\gamma_t|\le \la - \ell_\beta$, where $\beta = |\om_t|/|D|$, in view of the definition \eqref{minper}. Since the isoperimetric profile function $\alpha\to\la$ is lipschitz (see for instance \cite{ros}) we deduce that $\la -  |\gamma_t|\le C \left|\alpha |D| - |\om_t|\right|.$ From \eqref{texists} there exists $t_0$ such that
$\left|\alpha |D| - |\om_{t_0}|\right| \le C\tep\ltep,$
therefore for any $t\in I_\ep$ we have
$$\delta_-(t) \le C \left|\alpha |D| - |\om_t|\right| \le \left|\alpha |D| - |\om_{t_0}|\right| + \left| |\om_{t_0}|- |\om_{t}|\right|\le C\tep\ltep + \left|\om_{\mep - C\ltep}\setminus \om_{C\ltep}\right|.$$
Together with \eqref{almostarea} we deduce that
$$\left|\om_{\mep - C\ltep}\setminus \om_{C\ltep}\right| \le C\tep\(\ltep + \Dep + \left|\om_{\mep - C\ltep}\setminus \om_{C\ltep}\right|\),$$
which implies \eqref{smallarea} and the lemma if $\tep$ is small enough.
\end{proof}

\subsection{Proof of Theorem~\ref{enloc}} Assume the hypothesis of Theorem~\ref{enloc} are satisfied. We assume that $\Dep = o(\mep\la)$ as $\ep\to 0$ otherwise the  conclusion is trivial. Then, as is well-known in the scalar case since Modica-Mortola \cite{momo} and in this case from Aftalion-Royo Letellier \cite{AAJRL2013}, any sequence $\{\ep\}$ converging to zero admits a subsequence (not relabeled)   such that $\{(\riep, \rjep)\}_\epp$ converges to $(\chi_{\om_\alpha}, \chi_{\om_\alpha^c})$,  where $\om_\alpha$ is a minimizer of \eqref{minper}. We consider such a subsequence, for which
\beq\label{bvcv} \rie \to \chi_{\om_\alpha}, \rje \to \chi_{\om_\alpha^c}\eeq
weakly in $BV$, and strongly in $L^1$.

We wish to prove that for any $\eta>0$, denoting $V_\eta$ an $\eta$-neighbourhood of  $\gamma_\alpha := \partial\om_\alpha\cap D$,  there exists $C>0$ such that if $\epp$ is small enough depending on $\eta$ we have
\beq\label{lowb} F_\ep(\riep, \rjep,V_\eta) \ge \mep\la  - C\(\Dep + \ltep\) .\eeq
(Note that the second assertion in \eqref{uppb} follows immediately from the above and \eqref{hyptheo}). We begin by proving
\begin{lemm} \label{diffsym} Assume $\alpha\in (0,1)$ and let $\om_\alpha$ be a minimizer of \eqref{minper}. Then for any $\delta>0$ there exists $\eta>0$ such that if $\om_{\alpha'}$ is a minimizer for \eqref{minper} with $|\alpha - \alpha'| < \eta$ and if $|\om_\alpha \triangle \om_{\alpha'}| < \eta$, then
$$\om_{\alpha'}\subset \om_\alpha + B_\delta.$$
\end{lemm}

\begin{proof} Assume by contradiction that there exists $\delta>0$ and a sequence $\{\omega_{\alpha_n}\}_n$ of minimizers of \eqref{minper} such that $\alpha_n\to\alpha$ and $|\om_{\alpha_n}\triangle\om_\alpha|\to 0$. Then every subsequence of $\{\chi_{\omega_{\alpha_n}}\}_n$ has a subsequence which weakly converges in $BV$. But the only possible limit is $\chi_{\omega_\alpha}$ since $|\om_{\alpha_n}\triangle\om_\alpha|\to 0$. Therefore the whole sequence converges to $\chi_{\omega_\alpha}$ weakly in $BV$.

The result then follows from the regularity of sets with minimal perimeter (see for instance the book by Giusti \cite{giusti}, or the recent notes by Cozzi and Figalli \cite{cozzifigalli}).
\end{proof}

\begin{proof}[Proof of Theorem~\ref{enloc}] Let
$$\om_t = \{\dep\circ u_\ep  < t\},\quad \alpha(t) = \frac{|\om_t|}{|D|},\quad \ell(t) = \ell_{\alpha(t)},\quad  I_\ep = [C\ltep, \mep - C\ltep],$$
where $\dep$ is the distance defined in \eqref{dv} choosing as potential the function $\W$ defined in \eqref{vep}, and where $\la$ is defined in \eqref{minper}.

We have from \eqref{hyptheo} that
\beq\label{bd1th}\int_0^{\mep}\per_D(\om_t)\,dt \le F_\ep(u) \le \mep\la  + \Dep,\eeq
therefore
$$\int_{I_\ep}\per_D(\om_t)\,dt \le  \int_{I_\ep} \la\,dt +\Dep + C\ltep.$$
From \eqref{hyptheo} we ma apply Lemma~\ref{lemmarea} to $(|\rie|,|\rje|)$ to find that for every $t\in I_\ep$ we have
\beq\label{goodarea}|\om_t - \alpha|D||\le C\tep (\Dep+\ltep), \eeq
which implies, since  $\alpha\to \la$ is Lipschitz, that
\beq\label{sandwich}\per_D(\om_t)\ge \ell_{\alpha(t)}\ge \la - C\tep \Dep+\ltep).\eeq
Together with \eqref{bd1th} this yields
\beq\label{thstep1}\int_{I_\ep}\left|\per_D(\om_t) - \la\right|\,dt \le C\(\Dep + \ltep\).\eeq

We now make use of the localisation of perimeter proved in Proposition~\ref{locper}. According to Proposition~\ref{locper}, for any $t\in I_\ep$ there exists $\tomt$ which minimizes \eqref{minper} for $\alpha(t)$ such that
\beq \label{pernearmin} \ell(\partial\om_t\cap V_{\delta,t}^c) \le C(\per_D(\om_t) - \ell_{\alpha(t)}),\eeq
where $V_{\delta,t}$ denotes a $\delta$-neighbourhood of $\partial\tomt\cap D$. This implies in particular the existence of $C>0$ such that if we choose $\eta>0$, then for any  $t\in I_\ep$ we have
\beq\label{triangle1}\per_D(\om_t) - \ell_{\alpha(t)} \le \eta \quad\Longrightarrow\quad |\om_t\triangle\tomt| < C\eta.\eeq
Using Lemma~\ref{Wd} there exists $\beta>0$ independent of $\ep$ such that $\om_{C\ltep}\subset {u_\ep}^{-1}(B(a,\beta))$. But from  Lemma~\ref{lemmarea} we have that $|\om_{C\ltep}|$ converges  to $|\om_\alpha|$ as $\ep\to 0$ while from \eqref{bvcv} we have that $|\om_\alpha\triangle {u_\ep}^{-1}(B(a,\beta))|$ converges to $0$ as $\ep\to 0$. It follows that if $\ep$ is small enough then
\beq \label{triangle2}|\om_{C\ltep}\triangle\om_\alpha|\le \eta,\eeq
Using  Lemma~\ref{lemmarea} we have also that $|\om_t\triangle \om_{C\ltep}|<\eta$ for small $\ep$, which together with \eqref{triangle1} and \eqref{triangle2} implies that given $\eta>0$, if $\ep>0$ is small enough then  for any $t\in I_\ep$
$$|\tomt\triangle\om_\alpha| < C\eta.$$
In view of  Lemma~\ref{diffsym}, if choosing $\eta$ small enough we deduce that when $\ep$ is small enough
\beq\label{altern}\per_D(\om_t) - \ell_{\alpha(t)} \le \eta \quad\Longrightarrow\quad \partial\tomt\cap D\subset V_\eta.\eeq

Combining \eqref{pernearmin} and \eqref{altern}, we thus proved that $\forall t\in I_\ep$,
$$ \per_D(\om_t) - \ell_{\alpha(t)}\le \eta \Longrightarrow \per_{V_{2\delta}^c}(\om_t) \le C(\per_D(\om_t) - \ell_{\alpha(t)}).$$
Let $T$ denote the set of $t$'s such that $\per_D(\om_t) - \ell_{\alpha(t)}> \eta$ is bounded above by $(\Dep +C\ltep)/\eta$. We have
\begin{equation*}
\begin{split}
\int_{I_\ep}\per_{V_{3\delta}}(\om_t)\,dt &\ge  \int_{I_\ep}\per_{D}(\om_t) - \per_{V_{2\delta}^c}(\om_t)\,dt\\
&\ge  \int_{t\in I_\ep \setminus T}\ell_{\alpha(t)} - C(\per_D(\om_t) - \ell_{\alpha(t)})\,dt.
\end{split}
\end{equation*}
Using \eqref{thstep1} and \eqref{sandwich} we find
$$\int_{I_\ep}\per_{V_{3\delta}}(\om_t)\,dt \ge  \int_{t\in I_\ep\setminus T}\la\,dt - C(\Dep + \ltep).$$
But \eqref{thstep1} also  implies that the measure of the set $T$ of $t$'s such that $\per_D(\om_t) - \ell_{\alpha(t)}> \eta$ is bounded above by $(\Dep +C\ltep)/\eta$, therefore
$$\int_{I_\ep}\per_{V_{3\delta}}(\om_t)\,dt \ge  \int_{I_\ep}\la\,dt - \frac\la\eta(\Dep + \ltep) - C(\Dep + \ltep).$$
The left-hand side being bounded above by $F_\ep(u, V_{3\delta})$ we deduce that $$F_\ep(u, V_{3\delta})\ge \mep\la - C(\Dep + \ltep),$$
proving \eqref{lowb} and the theorem.
\end{proof}

%%%%%%%%%%%%%%%%%%%%%%%%%%%%%%%%%%%%%%%%%%%%%%%%%%%%%%%%%%%%%%%%%%%%%%%%%%%%%%%%%%%%%%%%%%%%%%%%%

\section{Proof of Theorems \ref{upstripes} and  \ref{theolong}}

\subsection{Proof of Theorem \ref{upstripes}}

The idea
 is that the modulus of the wave functions are invariant in the $y$ direction and will only depend on the
 $x$ variable, while the gradient of the phase has a staircase like increase
  in the $x$ direction. This construction is inspired from the test function of \cite{KT}.

We want to use a small scale  to build an upper bound  with stripes at this scale. We will assume that
 $u_1^2+u_2^2=1$. The natural scale is therefore proportional to $\te$. We set $\Omega=\lambda/\te^2$ and the scale $b_\ep=\mu/\sqrt\Omega=\mu \te/\sqrt \lambda$.
 We define $v_k(x)$ on a square $K$ of size 1 to be $u_k(b_\ep x)$. Therefore the rescaled energy on $K$ is
 \beq\label{Erescaled}
 E_1(v_1,v_2)=\int_{K} \sum_{k=1,2} \frac 12 |\nabla v_k-i \mu^2 x^\perp v_k|^2+\frac {\mu^2}{2\lambda} |v_1|^2  |v_2|^2.\eeq The upper bound for our full energy is then,
 for a well-chosen center of the grid,
 \beq\label{ubener} \frac{|D|}{b_\ep^2} \min E_1.\eeq
  We define $\rho_k=|v_k|$, $v_k=\rho_k e^{i\varphi_k}$ and $j_k=\nabla \varphi_k -\mu^2 x^\perp$. We will assume that neither $\rho_k$ nor $j_k$ depends on $y$, and that they are both $1$-periodic with respect to $x$. We will look for $\rho_1$ such that
  $$\rho_1(0)=\rho_1(1)=0, \ \rho_1 (1/2)=1,\ \rho_1 \quad\hbox{is even with respect to } 1/2.$$
  Moreover, since $\rho_1^2+\rho_2^2=1$, then ${\rho_1'}^2+{\rho_2'}^2={\rho_1'}^2/(1-\rho_1^2)$.

  Since the ground state of $E_1$ satisfies $\diver (\rho_k^2 j_k)=0$, this implies that $\rho_k^2 j_{k,x}$ is constant and we can set this constant equal to 0, which implies that $j_{k,x}=0$. Then, since $\curl j_k =-2 \mu^2$,  we have
  $$\frac{\partial j_{k,y}}{\partial x}=2 \mu^2.$$ We point out that $j_{k,y}$ can only have a jump where $\rho_k$ vanishes. This yields, for $x\in[0,1)$,
  $$ j_{1,y}=2 \mu^2 (x-\frac 12),$$
  $$j_{2,y}=2 \mu^2 x \hbox{ for } x\in (0,\frac 12),\ j_{1,y}=2 \mu^2 (x- 1)  \hbox{ for } x\in (\frac 12,1).$$ Since $\rho_2^2=1-\rho_1^2$, we find that
  $$\frac12\sum_{k=1,2}\int_0^1 \rho_k^2j_k^2=4 \mu^4\int_0^{1/2}(1-\rho_1^2) x^2+\rho_1^2 (x-\frac 12)^2= \mu^4 \(\frac 16+\frac\alpha 2 - 4 \int_0^{1/2}x\rho_1^2\).$$
  The energy $E_1$ of our test function is therefore
  \beq\label{testener} \mu^4 \(\frac 16+\frac\alpha 2\)+\frac 12 \int_0^1 \frac{{\rho_1'}^2}{(1-\rho_1^2)}+\frac{\mu^2}{\lambda}\rho_1^2 (1-\rho_1^2) - 4\mu^4 x\rho_1^2.\eeq We make a change of function $\rho_1=\sin \theta$ and recall (\ref{ubener}), which yields as an upper bound
  $$|D|\Omega \left (\(\frac 16+\frac\alpha 2 - 4 \int_0^{1/2}x\sin^2\theta\) \mu^2+\frac 1 {\mu^2} \int_0^{1/2}{\theta '}^2+\frac 1{4 \lambda} \int_0^{1/2} \sin^2 2\theta\right ).$$ This yields the Theorem.

 % We construct a one variable, periodic
 %  regular function on bands of size 1, with value 0 at 0 and 1, which is equal to 1 between $1/2 - \eta /2$ and
 %  $1/2+\eta /2$, and prescribed $L^2$ norm. We call $\rho_1(x)$ the square of this function. We can always assume that $\rho_1$ is $C^2$ and $\rho_1 (x)+\rho_1 (x+1/2)=1$.
  % Then, for the phase, we define $\theta_1 (x,y)=-\O xy$ so that it solves
  % $$div (\rho_1 (\nabla \theta_1 -\O x^\perp))=0$$
  %  with Neumann boundary condition on the band $(0,1)\times \R$. Note that the $x^\perp$ function is not periodic, hence the
   % staircaselike increase. Then we rescale the band to have the test function on a band of size $b_\eps$ and define $u_1(x,y)=\sqrt{\rho_1 (x/b_\ep)}e^{i\theta_1(x/b_\ep,y/b_\ep)}$.
   %  To match neighboring bands, there is no problem since the wave functions are  0 on the  boundary and the minimization is in $H^1$. In order to construct the wave function $u_2$, we define
   %  $u_2(x,y)=u_1(x+b_\ep/2,y)e^{-i\O b_\ep y}$.

   %  Estimating the energy density of all terms, we find it is of order
    % $$\frac 1{b_\ep^2} +\O^2 b_\ep^2+ \frac {\delta -1}{\ep^2}.$$ The hypothesis (\ref{eqbe2}) implies that the energy density is of order $1/b_\ep^2$. Therefore, on a typical square of size $R$, the energy of this test function is bounded by $CR^2/b_\ep^2$ and, on a square of size $Rb_\ep$,  bounded by $CR^2$.
    %  This provides an upper bound for the minimal energy.

\subsection{Proof of Theorem  \ref{theolong}}
      Let $(u_{1,\ep}, u_{2,\ep})$ be a minimizer of $E_\ep$ and let
      $E_\ep(K(x,R\tep))$ be the energy of $(u_{1,\ep}, u_{2,\ep})$ integrated on the square
       $K(x,R\tep)$.

 Because of Theorem~\ref{upstripes}, we have a bound for $E_\ep(K(x,R\tep))$ of the order of  $R^2$ for almost all squares, in the sense that
 for all $\eta>0$, there exists a constant $C$ such that
 \beq\label{eqeta}
 |\{ x \hbox{ s.t. } E_\ep(K(x,R\tep))>CR^2 \} |<\eta\eeq
 We are going to prove that each $x$ such that $E_\ep(K(x,R\tep))<CR^2$ it holds that
 $$\dashint_{K(x,R\tep)} |u_1|^2>\alpha \hbox{ and } \dashint_{K(x,R\tep)} |u_2|^2>\alpha,$$
 where $\alpha$ depends only on $C$ and $\lambda$.
 The conclusion of the Theorem will follow from (\ref{eqeta}) and this claim.

 We are going to prove the claim by contradiction. Assume that
  $$E_\ep(K(x,R\tep))<CR^2$$ and $\dashint_{K(x,R\tep)} |u_1|^2<\alpha$.
  From the energy bound, we infer that
  $$\dashint_{K(x,R\tep)} (1-|u_1|^2-|u_2|^2)^2<C\frac {\ep^2}{\tep^2}.$$ Therefore, for $\ep$ small,
  \beq \label{newu}\dashint_{K(x,R\tep)} |1-|u_1|^2-|u_2|^2|<\alpha \hbox{ and }\dashint_{K(x,R\tep)} |1-|u_2|^2|<2\alpha.\eeq

  We rescale by $R\tep$ the length and call $\tilde u_i$ the rescaled functions. Then, the rescaled rotation being defined as  $\tilde \O=R^2 \tep^2\O=R^2\lambda$, from the energy bound and  (\ref{newu}) we deduce that on the rescaled square $K$ of sidelength $1$ we have
  \beq\label{newE}\int_{K} |\nabla \tilde u_2-i\tilde u_2 \tilde \Omega \times r|^2+\frac 1 \alpha |1-|\tilde u_2|^2|<C.\eeq

  On the other hand we may use  the energy estimates and vortex constructions from the Ginzburg-Landau theory \cite{SSbook}, using $\sqrt\alpha$ as the Ginzburg-Landau parameter. These imply, after considering the different possible cases  $\tilde \O <C \log\alpha$, $\log\alpha \ll  \tilde \O \ll 1/\alpha$ and $1/\alpha < C\tilde\O$ that the following lower-bound holds, where $c>0$ is universal, if $\tilde\O>1$, $\alpha<1$:
  $$\int_{K} |\nabla \tilde u_2-i\tilde u_2 \tilde \Omega \times r|^2+\frac 1 \alpha |1-|\tilde u_2|^2|\geq c{\tilde \Omega}.$$
Recalling that $\tilde \O=R^2 \tep^2\O=R^2\lambda$, we deduce that if $R$ is sufficiently large, and $\alpha$ small, this lower-bound contradicts (\ref{newE}) and the claim holds.
Note that we could also not resort to Ginzburg-Landau theory and simply argue that the left-hand side of \eqref{newE} must be large if $\tilde\O$ is large enough and $\alpha$ is small enough by a compactness argument: assume there exists $\tilde\O_n\to +\infty$, $\alpha_n\to 0$ and $u_{2,n}$ satisfying \eqref{newE}, then obtain a contradiction.

% $$|\{ x,\ E_\eps (K_{x,R})<R^2A\}|>|D| (1-\frac 1 {|A|}),$$ so that
 %$$|\{ x,\ E_\eps (K_{x,R})>R^2A\}|>\frac {|D|}{ |A|}.$$

% {\bf Step 2:} GL type analysis with $\eps _{eff}=\eta$.
% $$\min E_\eps (K_{x,R}) \geq R^2 A_\eta\hbox{ where }
% A_\eta=\min (C\O^2 R^2 , \O \ln \frac 1 {\eta \sqrt \O}).$$

%{\bf Step 3:} This implies that$ |\{ x, \int u_1^2>1-\eta\} <\frac {|D|}{A_\eta}$

%{\bf Step 4:} Conclusion: in most squares, there is coexistence of $u_1$
% and $u_2$: there exists $\eta_0 >0$ and $R_0>0$ such that   $ |\{ x, \int u_1^2>1-\eta_0\} <\frac {|D|}{2^{10}}$

  %%%%%%%%%%%%%%%%%%%%%%%%%%%%%%%%%%%%%%%%%%%%%%%%%%%%%%%%%%%%%%%%%%%%%%%%%%%%%%%%%%%%%%%%%%%%%%%%%

%%%%%%%%%%%%%%%%%%%%%%%%%%%%%%%%%%%%%%%%%%%%%%%%%%%%%%%%%%%%%%%%%%%%%%%%%%%%%%%%%%%%%%%%%%%%%

%%%%%%%%%%%%%%%%%%%%%%%%%%%%%%%%%%%%%%%%%%%%%%%%%%%%%%%%%%%%%%%%%%%%%%%%%%%%%%%%%%%%%%%%%%%

\end{document}